\documentclass[11pt, letterpaper, oneside]{article}

\usepackage{latexsym, amsmath, amssymb, amsfonts, amscd}
\usepackage{lmodern}
\usepackage{nicefrac,mathtools,enumitem}
\usepackage{microtype}
\usepackage{amsthm}
\usepackage{t1enc}
\usepackage[mathscr]{eucal}
\usepackage{indentfirst}
\usepackage{graphicx}
\usepackage{fancyhdr}
\usepackage{fancybox}
\usepackage{psfrag}
\usepackage[all,poly,web,knot]{xy}
\usepackage{epsfig,diagrams}
\usepackage{enumitem}

\theoremstyle{plain}
\newtheorem{thm}{Theorem}[section]

\newtheorem{prop}[thm]{Proposition}
\newtheorem{lemma}[thm]{Lemma}
\newtheorem{cor}[thm]{Corollary}

\newtheorem{assump}[thm]{Assumptions}

\renewcommand{\latticebody}{\drop@{}}

\theoremstyle{definition}
\newtheorem{defi}[thm]{Definition}
\newtheorem{pdef}[thm]{Proposition-Definition}

\theoremstyle{remark}
\newtheorem{remark}[thm]{Remark}

\newcommand{\la}{\leftarrow}
\newcommand{\ra}{\rightarrow}
\newcommand{\lra}{\longrightarrow}
\newcommand{\lla}{\longleftarrow}
\newcommand{\rra}{\Rightarrow}

\newcommand{\thla}{\twoheadleftarrow}
\newcommand{\thra}{\twoheadrightarrow}

\newcommand{\N}{\ensuremath{\mathbb N}}

\newcommand{\R}{\ensuremath{\mathbb R}}

\renewcommand{\k}{\ensuremath{\mathfrak{k}}}

\renewcommand{\d}{d}

\newcommand{\D}{\mathscr D}

\newcommand{\cA}{\mathcal{A}}
\newcommand{\cC}{\mathcal{C}}            

\newcommand{\cX}{\mathcal{X}}
\newcommand{\cY}{\mathcal{Y}}
\newcommand{\cZ}{\mathcal{Z}}
\newcommand{\cG}{\mathcal{G}}
\newcommand{\cH}{\mathcal{H}}
\newcommand{\cT}{\mathcal{T}}             

\newcommand{\bepsilon}{\mbox{\boldmath $\epsilon$}}
\newcommand{\bDelta}{\mbox {\boldmath $\Delta$}}
\newcommand{\equivalence}{hypercover}
\newcommand{\equivalences}{hypercovers}

\newcommand{\Equivalences}{Hypercovers}

\DeclareMathOperator{\pr}{pr} 
\DeclareMathOperator{\id}{id}

\DeclareMathOperator{\Sk}{Sk}
\DeclareMathOperator{\Cosk}{Cosk}

\newcommand{\tG}{\tilde{G}}

\newcommand{\tU}{\tilde{U}}

\newcommand{\teta}{\tilde{\eta}}

\newcommand{\be}{\bar{e}}
\newcommand{\bg}{\bar{g}}
\newcommand{\bareta}{\bar{\eta}}
\newcommand{\bgamma}{\bar{\gamma}}
\newcommand{\bphi}{\bar{\phi}}

\newcommand{\bt}{\mathbf{t}}                  
\newcommand{\bs}{\mathbf{s}}                  
\newcommand{\bbt}{\bar{\mathbf{t}}}           
\newcommand{\bbs}{\bar{\mathbf{s}}}           

\def\U{{\mathcal U}}
\def\L{\Lambda}
\def\D{\Delta}
\newarrow{into} C--->
\newarrow{Eq} =====
\newarrow{dashto} ....>

\def\pD{\partial\D}
\def\lht{$\infty$-groupoid objects in $(\cC, \cT)$}
\def\ngpd{$n$-groupoid objects in $(\cC, \cT)$}

\def\PB(#1,#2,#3,#4){
\left\{\begin{matrix}#1&\!\!\!\stackrel{?}{\longrightarrow}&\!\!\!#2\\
\downarrow&&\!\!\!\downarrow\\
#3&\!\!\!\stackrel{?}{\longrightarrow}&\!\!\!#4\end{matrix}\right\}}

\def\pb(#1,#2,#3,#4){ \hom(#1 \to #3, #2 \to #4)}

\newtheorem{df}[thm]{Definition}

\newcommand \Kan {\mathit {Kan}}
\newcommand \Lmor {\mathit {Lmor}}
\newcommand \yon {\mathit {yon}}
\newcommand \Pb {\mathit {PB}}
\newcommand \an {a } 
\renewcommand \id {\mathit {id}}
\renewcommand \pr {\mathit {pr}}
\newcommand \sk {\mathit {sk}}
\renewcommand \exp {\mathit {exp}}
\newcommand \op {\mathit {op}}

\begin{document}

\title{ $n$-groupoids and stacky groupoids}
\author{Chenchang Zhu \thanks{Research partially supported by the Liftoff fellowship 2004 of the Clay
Institute}\\
Mathematisches Institut\\
Bunsenstr. 3-5 D-37073 G\"ottingen Germany \\
  \small{(zhu@uni-math.gwdg.de) }}
\date{\today}

\maketitle

\begin{abstract}
We discuss two generalizations of Lie groupoids. One consists of
Lie $n$-groupoids defined as simplicial manifolds with trivial
$\pi_{k\geq n+1}$. The other consists of stacky Lie groupoids
$\cG\rra M$ with $\cG$ a differentiable stack. We build a $1$--$1$
correspondence between Lie $2$-groupoids and stacky Lie groupoids up
to a certain Morita equivalence. We prove this in a general set-up so
that the statement is valid in both differential and topological
categories.  \Equivalences \ of higher
groupoids in various categories are also described.
\end{abstract}

\noindent KEY WORDS \quad stacks, groupoids, simplicial objects,
Morita equivalence

\tableofcontents

\section{Introduction}

Recently there has been much interest in higher group(oid)s, which
generalize the notion of group(oid)s in various ways. Some of them
turn out to be unavoidable to study problems in differential
geometry. An example comes from the string group, which is a
$3$-connected cover of $Spin(n)$. More generally, to any compact
simply connected group $G$ one can associate its string group
$String_G$. It has various models, given by Stolz and Teichner
\cite{stolz,st} using an infinite-dimensional extension of
$G$, by Brylinski \cite{bry-mc1} using a $U(1)$-gerbe with the
connection over $G$, and recently by Baez et al.
\cite{baez:str-gp} using Lie $2$-groups and Lie $2$-algebras.
Henriques \cite{henriques} constructs the string group as a higher
group that we study in this paper and  as an integration object of
a certain Lie $2$-algebra with an integration procedure which is also
studied in \cite{getzler, s:funny, z:lie2}.

Other examples come from a kind of \'etale stacky groupoid (called a Weinstein
groupoid) \cite{tz} built upon the very important work of \cite{cafe, cf}.
These stacky groupoids are the  global objects  in $1$--$1$ correspondence with  Lie
algebroids. A Lie algebroid can be understood as a degree-$1$ super
manifold with a degree-$1$ homological vector field, or more precisely as
a vector bundle $A\to M$ equipped with a Lie bracket $[, ]$ on the
sections of $A$ and a vector bundle morphism $\rho: A \to TM$,
satisfying a Leibniz rule, 
\[ [X, fY] = f[X, Y] + \rho(X) (f) Y.\]
When the base $M$ is a point, the Lie algebroid becomes a Lie algebra.
Notice that unlike (finite-dimensional) Lie algebras which always have associated Lie groups,
Lie algebroids do not always have  associated Lie groupoids
\cite{am1, am2}. One needs to enter the world of stacky groupoids
to obtain the desired $1$--$1$ correspondence. Since Lie algebroids are
closely related to Poisson geometry, this result applies to
complete the first step of Weinstein's program of quantization of
Poisson manifolds: to associate to Poisson manifolds their
symplectic groupoids \cite{w-poisson, wx}. It turns out that some
``non-integrable'' Poisson manifolds cannot have symplectic (Lie)
groupoids. A solution to this problem is given in \cite{tz2} with the above
result so that every Poisson manifold has a corresponding 
stacky symplectic groupoid.

2-group(oid)s were already studied in the early twentieth
century by Whitehead and his followers under various terms, such
as crossed modules. They are  also studied from the aspect of
``gr-champ'' (i.e. stacky groups) by Breen \cite{breen}. Recently,
various versions of 2-groups, with different strictness, have been
studied by Baez's school (the best thing is to read their
n-category caf\'e on http://golem.ph.utexas.edu/category/). These authors also
study 
a lot of developments on the subjects
surrounding 2-groups such as 2-bundles, 2-connections and the relation
with gerbes.  

It seems that it is required now to have a uniform method to
describe 2-groups so that it opens a way to treat all higher
groupoids. In this paper, we apply a simplicial method to describe all
higher groupoid objects in various categories in an elegant way, and
prove when $n=2$, they are the same as stacky groupoid objects in these category. This idea (set theoretically) was known much
earlier  by Duskin
and Glenn \cite{duskin, glenn}. 
The 0-simplices
correspond to the objects, the 1-simplices correspond to the arrows (or
1-morphisms), and the higher dimensional simplices correspond to the higher
morphisms. This method  becomes much more suitable when dealing with the differential or topological category.

Recall that a simplicial set (respectively manifold) $X$ is made
up of sets (respectively manifolds) $X_n$ and structure maps
\[ d^n_i: X_n \to X_{n-1} \;\text{(face maps)}\quad s^n_i: X_n \to X_{n+1} \; \text{(degeneracy maps)},\;\; \text{for}\;i\in \{0, 1, 2,\dots, n\} \]
that satisfy the coherence conditions
\begin{equation}\label{eq:face-degen}
\begin{split}
 d^{n-1}_i d^{n}_j = d^{n-1}_{j-1} d^n_i \; \text{if}\; i<j, &\quad s^{n}_i s^{n-1}_j = s^{n}_{j+1} s^{n-1}_i \; \text{if}\; i\leq j,
 \\
 d^n_i s^{n-1}_j = s^{n-2}_{j-1} d^{n-1}_i \; \text{if}\; i<j, &\quad
 d^n_j s^{n-1}_j=\id=d^n_{j+1} s^{n-1}_j,\quad
 d^n_i s^{n-1}_j = s^{n-2}_j d^{n-1}_{i-1} \; \text{if}\; i> j+1.
\end{split}
\end{equation}

The first two examples of simplicial sets are the simplicial $m$-simplex $\D[m]$ and
the horn $\L[m,j]$ with
\begin{equation}\label{eq:simplex-horn}
\begin{split}
(\D[m])_n & = \{ f: (0,1,\dots,n) \to (0,1,\dots, m)\mid f(i)\leq
f(j),
\forall i \leq j\}, \\
(\L[m,j])_n & = \{ f\in (\D[m])_n\mid \{0,\dots,j-1,j+1,\dots,m\}
\nsubseteq \{ f(0),\dots, f(n)\} \}.
\end{split}
\end{equation}
In fact the horn $\L[m,j]$ is a simplicial set obtained from the
simplicial $m$-simplex $\D[m]$ by taking away its unique
non-degenerate $m$-simplex as well as the $j$-th of its $m+1$
non-degenerate $(m-1)$-simplices, as in the following picture (in
this paper all the arrows are oriented from bigger numbers to
smaller numbers): \vspace{.6cm}

\centerline{\epsfig{file=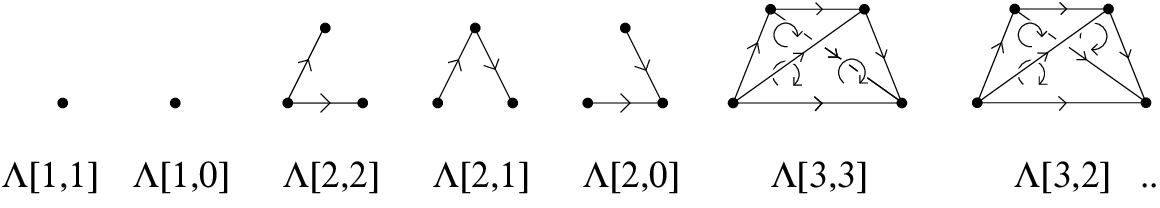,height=1.9cm}}
\vspace{.6cm}

A simplicial set $X$ is \emph {Kan} if any map from the horn
$\L[m,j]$ to $X$ ($m\ge 1$, $j=0,\dots,m$), extends to a map from
$\D[m]$. Let us call $\Kan(m,j)$ the Kan condition for the horn
$\L[m,j]$. A \emph {Kan simplicial set} is therefore a simplicial set
satisfying $\Kan(m,j)$ for all $m\ge 1$ and $0\leq j\leq m$. In the
language of groupoids, the Kan condition corresponds to the
possibility of composing and inverting various morphisms. For example, the
existence of a composition for arrows is given by the condition
$\Kan(2,1)$, whereas the composition of an arrow with the inverse
of another is given by $\Kan(2,0)$ and $\Kan(2,2)$. \vspace{.6cm}

\begin{equation}\label{compo}
\epsfig{file=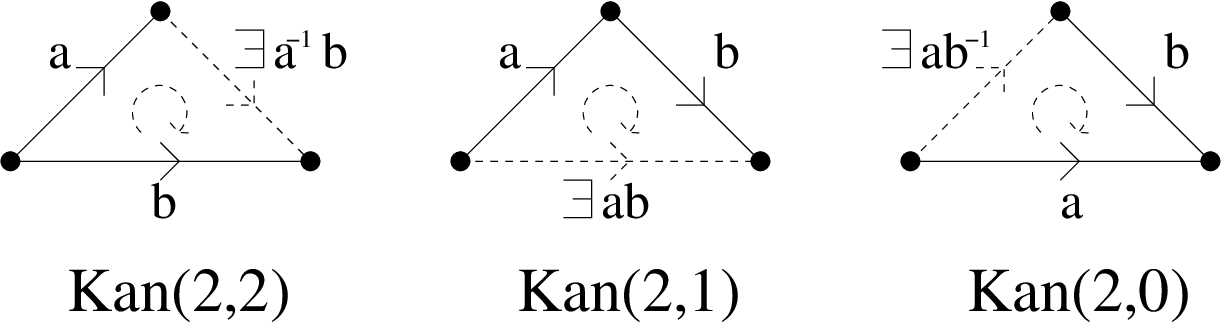,height=2cm}
\end{equation}

Note that the composition of two arrows is in general not unique,
but any two of them can be joined by a $2$-morphism $h$ given by
$\Kan(3,2)$. \vspace{.6cm}

\begin{equation}\label{composition}
\epsfig{file=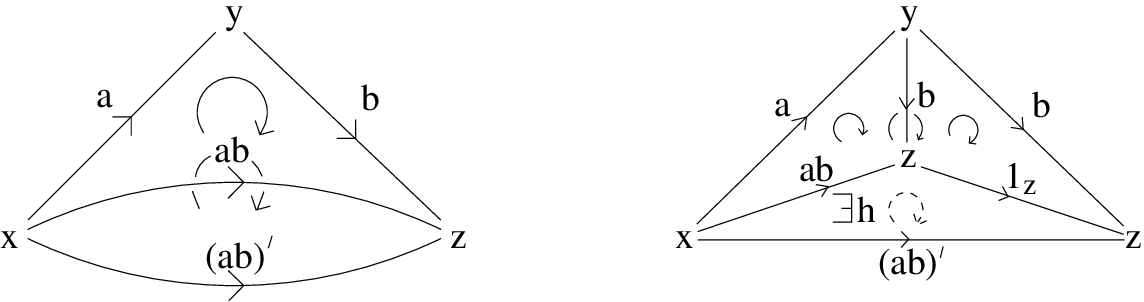,height=2.5cm}
\end{equation}

Here, $h$ ought to be a bigon, but since we do not have any bigons
in a simplicial set, we view it as a triangle with one of its
edges degenerate. The degenerate $1$-simplex above $z$ is denoted
$1_z$.

In an $n$-groupoid, the only well-defined composition law is the
one for $n$-morphisms. This motivates the following definition.

\begin{df}\label{groupoid}
An $n$-groupoid ($n\in\N \cup \infty$) $X$ is a simplicial set
that satisfies $\Kan(m,j)$ for all $0\le j\le m\ge 1$ and
$\Kan!(m,j)$ for all $0\le j\le m>n$, where\\
\noindent
\begin{tabular}{p{1.6cm}p{10cm}}
$\Kan(m,j)$:&  Any map $\L[m,j]\to X$ extends to a map $\D[m]\to X$.\\
$\Kan!(m,j)$: &
Any map $\L[m,j]\to X$ extends to a unique map $\D[m]\to X$.\\
\end{tabular}\\
\end{df}

An \emph {$n$-group} is an $n$-groupoid for which  $X_0$ is a point.
When $n=2$, they are different from the various kinds of
$2$-group(oid)s or double groupoids in \cite{bala:2gp,
br-sp} (see \cite[Appendix]{henriques:l-infty-v1} for an explanation of the
relation between our 2-group and the one in \cite{bala:2gp}), and are  not exactly
the same as in \cite{noohi}, as he requires a choice of composition and strict
units; however, they are the same as in \cite{duskin}.
 A usual groupoid (category with
only isomorphisms) is equivalent to a $1$-groupoid in the sense of
Def.\ \ref{groupoid}. Indeed, from a usual groupoid, one can
form a simplicial set whose $n$-simplices are given by sequences
of $n$ composable arrows. This is a standard construction called
the \emph {nerve} of a groupoid and one can check that it
satisfies the required Kan conditions.

On the other hand, a $1$-groupoid $X$ in the sense of Def.\
\ref{groupoid} gives us a usual groupoid with objects and arrows
given respectively by the $0$-simplices and $1$-simplices of $X$. The
unit is provided by the degeneracy $X_0\to X_1$, the inverse and
composition are given by the Kan conditions $\Kan(2,0)$, $\Kan(2,1)$
and $\Kan(2,2)$ as in (\ref{compo}), and the associativity is given
by $\Kan(3,2)$ and $\Kan!(2,1)$. \vspace{.6cm}

\centerline{\epsfig{file=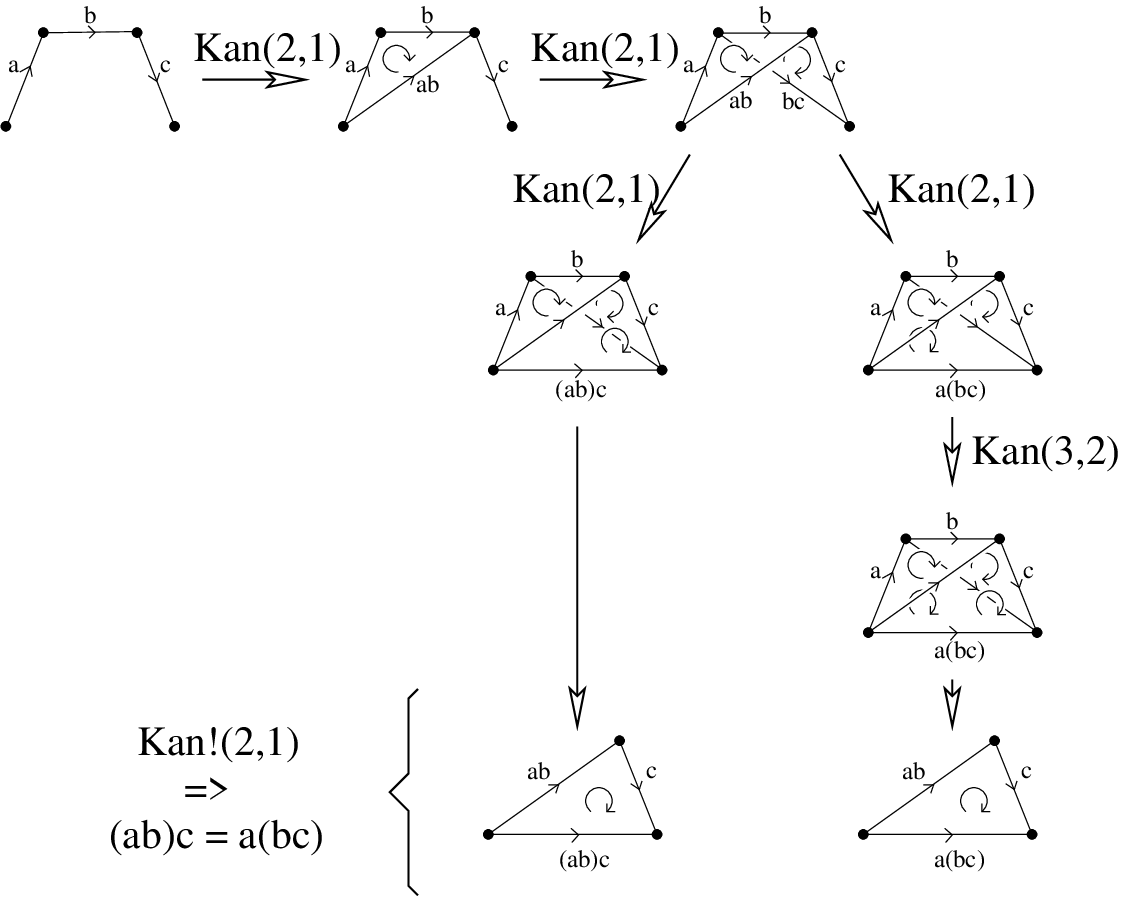,height=9cm}}
\centerline{Proof of associativity.} \vspace{.6cm}

This motivates the corresponding definition in a category $\cC$ with a
singleton Grothendieck pretopology $\cT$ which satisfies some additional mild
assumptions (see Assumptions \ref{assump:1}). We shall assume $\cC$ has
all coproducts.

\begin{defi}
A \emph {singleton Grothendieck pretopology}\footnote{The original
  definition of Grothendieck pretopology \cite{SGA4} requires
a collection of morphisms $U_i\to X$ for a cover. But since we assume that
$\cC$ has coproducts, a Grothendieck pretopology is given by a
singleton Grothendieck pretopology by 
declaring $\{U_i\to X\}$ to be a cover if $\Pi_{i}U_i \to X$ is a
cover. Hence when coproducts exist, these two concepts are the same. } on $\cC$ is a collection of morphisms, called covers, subject the following three axioms:
Isomorphisms are covers. The composition of two covers is a cover. If $U\to X$ is a cover and $Y\to X$ is a morphism, then the pull-back $Y\times_X  U$ exists, and the natural morphism $Y\times_X U \to Y$ is a cover.
\end{defi}

We list examples of categories equipped with singleton Grothendieck
pretopologies in Table \ref{table:1}, among which $(\cC_i, \cT'_i)$
for $i=1,2,3$ 
satisfy Assumptions \ref{assump:1} (with the terminal object $*$
being a point). 

\begin{table} [h!b!p!]
\caption{Categories and pretopologies } \label{table:1}
\begin{tabular}{|l|l|l|} 
\hline
Notation & $\cC$    & cover \\
\hline
$(\cC_1, \cT_1)$ & Banach manifolds\footnotemark\
and smooth morphisms &   surjective \'etale morphisms \\
\hline
$(\cC_1, \cT'_1)$ & Banach manifolds and smooth morphisms &  surjective submersions \\
\hline
$(\cC_2, \cT_2)$ & Topological spaces and continuous morphisms &surjective \'etale morphisms  \\
\hline
$(\cC_2, \cT'_2)$ & Topological spaces and continuous morphisms &   surjective continuous morphisms \\
\hline
$(\cC_3, \cT_3)$ & Affine schemes and smooth morphisms &surjective \'etale morphisms  \\
\hline
$(\cC_3, \cT'_3)$ & Affine schemes and smooth morphisms & surjective
smooth morphisms  \\ \hline
\end{tabular}
\end{table}
\footnotetext{See \cite[Section
  4]{henriques} for the convention on Banach manifolds that we use.}

\begin{df}\cite{henriques}\label{defngroupoids} 
An $n$-groupoid object $X$ ($n\in\N \cup \infty$) in $(\cC, \cT)$ is a simplicial
object in $(\cC, \cT)$ that satisfies $\Kan(m,j)$ for all $0\le j\le m\ge 1$ and
$\Kan!(m,j)$ $0\le j\le m>n$, where\\ \noindent
\begin{tabular}{p{1.6cm}p{10.6cm}}
$\Kan(m,j)$:& The restriction map $\hom(\D[m],X)\to\hom(\L[m,j],X)$
is a cover in $(\cC, \cT)$.\\
$\Kan!(m,j)$:& The restriction map
$\hom(\D[m],X)\to\hom(\L[m,j],X)$ is an isomorphism in $\cC$.\\
\end{tabular} \\
\end{df}

The notation $\hom(S, X)$, when $S$ is a simplicial set and $X$ is a
simplicial object in $\cC$, has the same meaning as in \cite[Section
2]{henriques}; in the case of a {\bf Lie $n$-groupoid} \cite[Def.\
1.2]{henriques}, which is an $n$-groupoid object in $(\cC, \cT)= (\cC_1,
\cT'_1)$, we can view simplicial sets $\D[m]$ and $\L[m,j]$ as
simplicial manifolds with their discrete topology so that $\hom(S,X)$
denotes the set of homomorphisms of simplicial manifolds with its
natural topology. Thus $\hom(\D[m],X)$ is just another name for $X_m$.
However it is not obvious that $\hom(\L[m,j],X)$ is still an object in
$\cC$, and it is a result of \cite[Corollary 2.5]{henriques} (see
Section \ref{sect:lie-n-gpd} for details). Moreover, \emph {a Lie
  $n$-group} is a Lie $n$-groupoid $X$ where $X_0=pt$.

On the other hand, a \emph {stacky Lie (SLie) groupoid} $\cG \rra
M$, following the concept of Weinstein (W-) groupoid in \cite{tz},
is a groupoid object in the world of differentiable stacks with
its base $M$ an honest manifold. When $\cG$ is also a manifold, $\cG \rra M$
is obviously a Lie groupoid.  \emph {W-groupoids}, which are \emph
{\'etale SLie} groupoids, provide a way to build the $1$--$1$
correspondence with Lie algebroids. This concept can be also
adapted to stacky groupoids in various categories (see Def.\ \ref{def:sliegpd}). 

Given these two  higher generalizations of Lie groupoids, Lie
$n$-groupoids and SLie groupoids, arising from different
motivations and constructions, we ask the following questions:
\begin{itemize}
\item Are SLie groupoids the same as Lie $n$-groupoids for some $n$?
\item If not exactly, to which extent they are the same?
\item Is there a way to also realize Lie $n$-groupoids as
integration objects of Lie algebroids?
\end{itemize}

In this paper, we answer the two first questions by

\begin{thm} \label{2-w}
There is a one-to-one correspondence between SLie (respectively
W-) groupoids and Lie $2$-groupoids (respectively Lie $2$-groupoids
whose $X_2$ is \'etale over $\hom(\L[2,j],X)$) modulo $1$-Morita
equivalences\footnote{Morita equivalences preserving $X_0$} of Lie
$2$-groupoids.
\end{thm}

The last question will be answered positively in a future work \cite{z:lie2}:

\begin{thm}\label{thm:2-a}
Let $A$ be a Lie algebroid and let $\Lmor(-,-)$ be the space of Lie
algebroid homomorphisms satisfying suitable boundary conditions.
Then
\[ \Lmor(T\Delta^2, A)/ \Lmor(T\D^3 , A) \Rrightarrow \Lmor(T\Delta^1, A) \rra
\Lmor(T\Delta^0, A), \] is a Lie $2$-groupoid corresponding to the
W-groupoid $\cG(A)$ constructed in \cite{tz} under the
correspondence in the above theorem.
\end{thm}

With a mild assumption about ``good charts'', we are able to prove a
stronger version of Theorem \ref{2-w} in various other categories, such as
topological categories (see
Theorem \ref{thm:1-1}). If we view a manifold as a set with additional
structure, then we can view our SLie groupoid $\cG \rra
M$ as a groupoid where the space $\cG$ of arrows is itself a category
with certain additional structure. From
this viewpoint,  our result is the
analogue in geometry of Duskin's result \cite{duskin2} in category
theory. Moreover, our stacks are required to be presentable by certain
charts in $\cC$. For example, when $(\cC, \cT) = (\cC_1, \cT'_1)$ the
differential category, our stacks are \emph {not} just categories
fibred in groupoids over $\cC_1$,  but furthermore can be
presented by Lie groupoids. They are called differentiable
stacks. Hence to prove our result, we use the equivalence of the $2$-category of differentiable
stacks, morphisms and $2$-morphisms and the $2$-category of Lie groupoids,
Hilsum--Skandalis (H.S.) bibundles  \cite{hs, m:bibundle}  and $2$-morphisms. This can be viewed as an
enrichment of Duskin's set-theoretical method. Then of course, this
enrichment requires a different approach and solutions of many
technical issues in geometry and topology that we prepare in Section \ref{sect:lie-n-gpd} and \ref{sect:sgpd}.  

Furthermore, a subtle point in the theory of stacks and groupoids is
that a stack can be presented by many Morita equivalent groupoids. Hence, for Theorem \ref{2-w} and \ref{thm:1-1}, we also develop the theory of \emph {morphisms} and
\emph {Morita
equivalence} of $n$-groupoids,  which  is expected to be useful in
the theory of $n$-stacks and $n$-gerbes and should correspond to
Morita equivalence of stacky groupoids in \cite{bz} when $n=2$.

The reader's first guess about the morphisms of \ngpd \  is probably
that a morphism $f:X\to Y$  ought to be a simplicial morphism, namely a collection of morphisms  $f_n:X_n\to Y_n$ in $\cC$ that commute with faces and
degeneracies. In the language of categories, this is just a
natural transformation from the functor $X$ to the functor $Y$. We
shall call such a natural transformation a \emph {strict map} from
$X$ to $Y$. Unfortunately, it is known that, already in the case
of usual Lie groupoids, such strict notions are not good enough.
Indeed there are strict maps that are not invertible even though
they ought to be isomorphisms. That's why people introduced the
notion of H.S. bibundles. Here is an
example of such a map: consider a manifold $M$ with an open cover
$\{\U_\alpha\}$. The simplicial manifold $X$ with
$X_n=\bigsqcup_{\alpha_1,\ldots,\alpha_n}\U_{\alpha_1}\cap\cdots\cap\U_{\alpha_n}$
maps naturally to the constant simplicial manifold $M$. All the
fibers of that map are simplices, in particular they are
contractible simplicial sets. Nevertheless, that map has no
inverse. 

The second guess is then to define a special class of strict maps
which we shall call \emph {\equivalences}. A map from $X$ to $Y$ would
then be a \emph {zig-zag} of strict maps
$X\stackrel{\sim}{\leftarrow}Z\to Y$, where the map $Z\stackrel{\sim}{\to} X$ is one of
these \equivalences. This will be equivalent to bibundle approach. The
notion of \equivalence \ is nevertheless very
useful (e.g., to define sheaf cohomology of \ngpd) and we study it
in Section \ref{sec:equi-lht}.

We also find some technical improvements of the concept of SLie
groupoid: it turns out that an SLie groupoid $\cG\rra M$ always
has a ``good groupoid presentation'' $G$ of $\cG$, which possesses
a strict groupoid map $M\to G$. Moreover the condition on the
inverse map can be simplified.\\
\\
\noindent {\bf Notation Chart}\\
$\cG \rra M$, $\cH \rra N$:  stacky groupoids; \\
$\bbs$, $\bbt$, $\bar{e}$, $\bar{i}$, $m$:  source,
target, identity, inverse  and multiplication of a stacky groupoid; \\
$G:=G_1\rra G_0$: a groupoid presentation of $\cG$; \\ 
$\bs_G$, $\bt_G$, $e_G$, $i_G$:   the source,
target, identity and inverse of the groupoid $G$ respectively; \\ 
$\bs, \bt: G_0 \to M$: the morphisms presenting $\bbs,\bbt: \cG \to M$
respectively; \\
$\eta_i$, $\gamma_i$: the face facing the vertex $i$, moreover
$\gamma_i$ belongs to $G_1$;\\
$\eta_{ijk}$, $\gamma_{ijk}$: the face with vertices $i$, $j$ and $k$,
moreover $\gamma_{ijk}$ belongs to $G_1$;\\
$J_l$, $J_r$: the left and the right moment maps\footnote{They are called
moment maps for the following reason: when we have a Hamiltonian action of a
Lie group $K$ on a symplectic manifold $E$ with a moment map $J: E
\to \k^*$, then the Lie groupoid $T^*
K \rra \k$ acts on $E$ with the help of the map $E \xrightarrow{J} \k^* $;  then
this result was generalized to any (symplectic) groupoid action in
\cite{mw} keeping the name ``moment map''.}  of an H.S. bibundle $E$
between two groupoid objects $K_1 \rra K_0$ and $K'_1 \rra K'_0$,  
\[
\xymatrix{ 
K_1 \ar[d] \ar@<-1ex>[d] & E \ar[dl]_{J_l} \ar[dr]^{J_r} & K'_1 \ar[d]
\ar@<-1ex>[d] \\
K_0 & & K'_0.
}
\]

\noindent {\bf Acknowledgments:} Here I would like to thank
Henrique Bursztyn and Alan Weinstein
for their hosting and very helpful discussions. I also thank Laurent
Bartholdi, Marco Zambon and Toby Bartels for many editing suggestions. I especially thank Andr\'e Henriques who pointed out to me the
potential correspondence of stacky groupoids and Lie $2$-groupoids
during the conference of ``Groupoids and Stacks in Physics and
Geometry'' in CIRM-Luminy 2004. I owe a lot to discussions with
Andr\'e. He contributed the exact definitions of  Lie $n$-groupoids
and their \equivalences,  and also nice pictures! Also I thank Ezra
Getzler very much for telling me the relation of his work
\cite{getzler},  $n$-groupoids and our work \cite{tz} during my trip
to Northwestern and his
continuous comments to this work later on. I thank John Baez and
Toby Bartels for discussions on Grothendieck pretopologies.  Finally, I thank the
referee a lot for much helpful advice.

\section{$n$-groupoid objects and morphisms in various categories} \label{sect:lie-n-gpd}

Lie groupoids and topological groupoids have been studied a lot
(see \cite{cw} for details). They are used to study foliations,
and more recently orbifolds, differentiable stacks and topological stacks
\cite{m-orbi, bx1,  noohi, gep-hen}. Here we will try to convince the reader
that it is fruitful to consider them within the context of
$n$-groupoid objects (Def.\ \ref{defngroupoids}), especially if one wants
to define and use sheaf cohomology. 

Our $n$-groupoid objects live in a category $\cC$ with a singleton Grothendieck
pretopology $\cT$ satisfying the following properties:
\begin{assump}\label{assump:1}
The category $\cC$ has a terminal object $*$, and for any object $X\in \cC$, the map $X\to *$ is a cover. 

The pretopology $\cT$ is subcanonical, which means that all the representable functors $T \mapsto \hom(T, X)$ are sheaves. 
\end{assump} 
\begin{remark} These properties are (\emph {only}) a part of Assumptions 2.2 in \cite{henriques}. It turns out that we do not need all the assumptions if we do not deal with further subjects, such as simplicial homotopy groups.
\end{remark}

As in \cite[Section 2]{henriques}, we sometimes talk about the limit of a diagram in $\cC$, before knowing its existence. For this purpose, we use the Yoneda functor
\[ 
\begin{split} 
\yon: \cC & \to \{ \text{Sheaves on $\cC$} \} \\ 
 X & \mapsto (T\mapsto \hom(T, X))
\end{split}
\]
to embed $\cC$ to the category of sheaves on $\cC$. Hence a limit of
objects of $\cC$ can always be viewed as the limit of the
corresponding representable sheaves using $\yon$. The limit sheaf is
representable if and only if the original diagram has a limit in
$\cC$.

\subsection{\Equivalences \ of $n$-groupoid objects} \label{sec:equi-lht}

First let us fix some notation of pull-back spaces of the form
$\Pb\big(\hom(A,Z)\to\hom(A,X)\leftarrow\hom(B,X)\big)$, where the
maps are induced by some fixed maps $A\to B$ and $Z\to X$. To
avoid the cumbersome pull-back notation, we shall denote these
spaces by
\[\PB(A,Z,B,X) \; \text{in the layout,} \quad \text{or} \; \hom(A\to B, Z\to X) \; \text{in the text.}\]
This notation indicates that the space parameterizes all commuting
diagrams of the form
$$\begin{matrix}A&\!\!\!\longrightarrow&\!\!\!Z\\
\downarrow&&\!\!\!\downarrow\\
B&\!\!\!\longrightarrow&\!\!\!X,\end{matrix}$$ where we allow the
horizontal arrows to vary but we fix the vertical ones.

\Equivalences \ of \ngpd \ are very much inspired by
hypercovers of \'etale simplicial objects \cite{SGA4, friedlander}  and by
Quillen's trivial fibrations for simplicial sets\footnote{In fact, $\infty$-groupoid objects in $(\cC, \cT)$ are called
\emph {Kan  simplicial objects} in $(\cC, \cT)$ \cite[Section
2]{henriques}.}  \cite{quillen:ha}.

\begin{defi}\label{defequivalence}
A strict map $f:Z\to X$ of \ngpd \ is \an \emph {\equivalence} if the natural
map from $Z_k=\hom(\D^k,Z)$ to the pull-back
\[\pb(\pD[k],Z, \D[k], X) =\Pb(\hom(
\partial \Delta[k] , Z)\ra \hom(\partial \Delta[k], X) \la X_k)\]
is a cover for $0\le k\le n-1$ and an isomorphism\footnote{When $n=\infty$, namely in the case of \lht, the requirement of isomorphism is empty.} for $k=n$. 
\end{defi}

But in our case, we need  Lemma
\ref{lemma:indct-mfd-equi} to \emph {justify} that $\pb(\pD[k],Z, \D[k], X)$
is an object in $\cC$ for $1\le k $ so that this definition makes
sense.  This is specially surprising since the spaces
$\hom(\pD[m],Z)$ need not be in $\cC$ (for example take $n=2$, $\cC$  the category of Banach manifolds, and
$Z$ the cross product Lie groupoid associated to the action of $S^1$
on $\R^2$ by rotation around the origin).  To simplify our notation, $\thra$ and $\thla$
always denote covers in $\cT$.

\begin{lemma}\label{lemma:indct-mfd-equi}
Let $S$ be a finite collapsible simplicial set\footnote{See
  \cite[Section 2]{henriques}.} of any dimension, and $T \mathrel{(\hookrightarrow} S)$ a
sub-simplicial set of dimension $\le m$. Let $f:Z\to X$ be a
strict map of \lht \  such that $\hom(\pD[l] \to \D[l], Z\to X) \in \cC$ for all $l \le m$ and the natural map\footnote{Since $Z_l=\hom(\D[l], Z)$ maps naturally to $ \hom(\D[l], X)$ and $ \hom(\pD[l],Z)$, there is a natural map from $Z_l$ to their fibre product $\pb(\pD[l], Z, \D[l], X)$.}
\[
Z_l\rightarrow\PB(\pD[l],Z,\D[l],X)
\]
is a cover for all $l\le m$. Then the pull-back
$\pb(T,Z,S,X)$ exists in $\cC$. Hence in particular, $\pb(\pD[m+1], Z,
\D[m+1], X)$ exists in $\cC$.
\end{lemma}

\begin{proof}
Let $T'$ be a sub-simplicial set obtained by deleting one $l$-simplex
from $T$ (without its boundary, and $T'\to T$ includes the case of
$\empty\to \Delta[0]$). We have a push-out diagram

\begin{diagram}
T'      &    \rTo    & \SWpbk T\\
 \uTo       &   & \uTo \\
\pD[l]&    \rinto    &\D[l].
\end{diagram}

Applying the functor $\pb(-, Z, S, X)$, this gives a
pull-back diagram
\begin{diagram}
\PB(T',Z,S,X)      &&   \lTo    & \raisebox{-.7cm}{\SWpbk} \PB(T,Z,S,X)&\phantom{=Z_l}\\ \\
 \dTo       &   && \dTo \\
\PB(\pD[l],Z,S,X)&&    \lTo    &\PB(\D[l],Z,S,X)&,
\end{diagram}

which may be combined with the pull-back diagram

\begin{diagram}
\PB(\pD[l],Z,S,X)      &&   \lTo    & \raisebox{-.7cm}{\SWpbk} \PB(\D[l],Z,S,X)\\ \\
 \dTo       &   && \dTo \\
\PB(\pD[l],Z,\D[l],X)&&    \lTo    &\PB(\D[l],Z,\D[l],X)&=Z_l
\end{diagram}

to give yet another pull-back diagram

\begin{equation}\label{bigpb}
\begin{diagram}
\PB(T',Z,S,X)      &&   \lTo    & \raisebox{-.7cm}{\SWpbk} \PB(T,Z,S,X)\\ \\
 \dTo       &   && \dTo \\
\PB(\pD[l],Z,\D[l],X)&&    \lTo    &\PB(\D[l],Z,\D[l],X)&=Z_l.
\end{diagram}
\end{equation}

By induction on the size of $T$ (\cite[Lemma 2.4]{henriques} implies the case when $T=\emptyset$) and the induction hypothesis, we may assume that the upper left and lower
left spaces in \eqref{bigpb} are known to be in $\cC$. The bottom
arrow is a cover by hypothesis. Therefore by the property of covers,
the upper right space is also in $\cC$, which is what we wanted to prove.
\end{proof}

As a byproduct of Lemma
\ref{lemma:indct-mfd-equi}, we have: 
\begin{lemma} \label{lemma:proj-sub}
If $Z\to X$ is \an \equivalence \ of \ngpd, then for a sequence of subsimplicial sets $T' \subset
  T \subset S$ where $S$ is collapsible, the natural map $\pb(T, Z, S, X) \to
\pb(T', Z, S, X)$ is a cover in $\cC$.  In particular,
\begin{enumerate}
\item \label{itm:x}  the natural map $\pb(\pD[m], Z, \D[m], X) \to
X_m$ is a cover, when we choose $T'=\emptyset$, $T=\pD[m]$ and $S=\D[m]$;
\item \label{itm:z} the natural map $Z_m \to \pb({\L[m,j]}, Z, \D[m], X) $
  is a cover in $\cC$, when we choose  $(T\to S)=(\D[m]\xrightarrow{\id}\D[m])$
  and $(T'\to S)=(\L[m,j]\hookrightarrow \D[m])$;
\item \label{itm:L} the natural map $\hom(\L[m,j], Z) \to \hom(\L[m,j],
  X)$ is a cover in $\cC$, when we choose $(T\to
  S)=(\L[m,j]\xrightarrow{\id} \L[m,j]) $ and $T'=\emptyset$;
\item \label{itm:n}  we have 
\begin{equation}\label{eq:ge-n}
Z_k \cong \pb(\pD[k],Z, \D[k], X), \quad \forall k\geq n. 
\end{equation}  
\end{enumerate}
\end{lemma}
\begin{proof} We use the same induction as in Lemma \ref{lemma:indct-mfd-equi} and only have to
notice that the lower lever map in \eqref{bigpb} is a cover, hence so
is the upper lever map. Since
composition of covers is still a cover, we obtain the result by
introducing a sequence of subsimplicial sets $T'=T_0\subset T_1 \subset \dots \subset T_{j-1} \subset
T_j $, where each $T_i$ is obtained from $T_{i-1}$ by removing a
simplex. 

For item \ref{itm:n}, we take $(T\to S)=(\pD[n+1]\hookrightarrow
\D[n+1])$ and $(T'\to S)=(\L[n+1, j] \hookrightarrow \D[n+1])$, and
use the fact that the lower lever map in \eqref{bigpb} is an
isomorphism when $l=n$. We obtain 
\[ \begin{split}
\pb(\pD[n+1], Z,
  \D[n+1], X) &\cong \pb({\L[n+1, j]}, Z,
  \D[n+1], X) \\ 
& = \hom(\L[n+1, j], Z) \times_{ \hom(\L[n+1, j], X)}
  X_{n+1} \cong Z_{n+1},
\end{split}
\] since $X_{n+1}\cong \hom(\L[n+1, j], X)$.  Then inductively, we obtain the result for all $k\ge n$.
\end{proof}

\begin{lemma} \label{lemma:comp-equi}
 The composition of \equivalences \ is still \an
\equivalence.
\end{lemma}
\begin{proof}
This is easy to verify, and we leave it to the reader.
\end{proof}

\begin{lemma}\label{lemma:ngpd-fp}
Given a strict map $f: Z \to X$ and \an \equivalence \ $f': Z' \to X$, the fibre product $Z \times_{X} Z'$ of \ngpd \ is still an $n$-groupoid
object in $(\cC, \cT)$.
\end{lemma}
\begin{proof}
We first notice that $Z\times_{X} Z'$ is a simplicial object (of sheaves on $\cC$) with
each layer $ \hom(\D[m],Z\times_{X} Z')=  Z_m \times_{X_m} Z'_m$. We
use an induction to show that  $Z\times_{X} Z'$ is an $n$-groupoid object
in $\cC$. First when $n=0$, $Z'_0\thra X_0$, hence $Z_0 \times_{X_0}
Z'_0 $ is representable in $\cC$ and $Z_0 \times_{X_0}
Z'_0 \thra * $.

Now assume that  $ \hom(\D[k],Z\times_{X} Z') \thra
\hom(\L[k,j],Z\times_{X} Z' ) $ is a cover in $\cC$ for $0\le j\le k <
m$.  By item
  \ref{itm:L} of Lemma \ref{lemma:proj-sub},  $\hom(\L[m,j], Z\times_{X} Z')$ is
representable.
When $m<n$, we need to show that $ \hom(\D[m],Z\times_{X} Z')
\thra  \hom(\L[m,j], Z\times_{X} Z') $ is a cover in $\cC$; when $m\ge
n$, we need to show that  $ \hom(\D[m],Z\times_{X} Z')
\cong  \hom(\L[m,j], Z\times_{X} Z') $ is an isomorphism in
$\cC$. When $m<n$,  applying $X_m \thra \hom(\L[m,j], X)$ to the south-east corner of the
following pull-back diagram in $\cC$, 
\[
\begin{diagram}
 \hom(\L[m,j], Z\times_{X} Z')  \hbox{\SEpbk}   && \rTo &
\hom(\L[m,j], Z') \\
\dTo & && \dTo  \\
\hom(\L[m,j], Z)  && \rTo & \hom(\L(m,j), X). \\
\end{diagram}
\]
By item \ref{itm:z} in Lemma \ref{lemma:proj-sub} and the fact that $Z$ is
an $n$-groupoid object in $(\cC, \cT)$, we have
\[
\begin{split}
Z'_m & \thra \pb({\L[m,j]}, Z', \D[m], X)= \hom(\L[m,j], Z')\times_{
  \hom(\L(m,j), X)}X_m, \\
Z_m  &  \thra  \hom(\L[m,j], Z). 
\end{split}
\]
Thus by Lemma \ref{lemma:fp}, we have that $ \hom(\D[m],Z\times_{X} Z')
\thra  \hom(\L[m,j], Z\times_{X} Z') $ is a cover in $\cC$, which
completes the induction. 
When $m\ge n$,  the three $\thra$'s
above becomes three $\cong$'s (see \eqref{eq:ge-n}). Hence  the same proof concludes $ \hom(\D[m],Z\times_{X} Z')
\cong  \hom(\L[m,j], Z\times_{X} Z') $.
\end{proof}

\begin{lemma}\label{lemma:inverse-comp-equi}
Given $Z$, $Z'$ and $X$ \ngpd, if   $f: Z \to X$ is \an \equivalence \ and  $Z''=Z \times_{X} Z'$ is
still an $n$-groupoid object in $(\cC, \cT)$,  the natural map  $Z'' \lra
Z'$ is \an \equivalence.
\end{lemma}
\begin{proof}
Apply $\hom(\pD[m] \to \D[m], -)$ to the pull-back diagram
\[
\begin{diagram}
\left\{\begin{matrix} Z'\times_X Z \\
\downarrow_{\pr_2} \\
Z\end{matrix}\right\}    \raisebox{-0.7cm}{\SEpbk}   &&   \rTo &  \left\{\begin{matrix} Z' \\
\downarrow_{f'} \\
X\end{matrix}\right\}    \\ \\
 \dTo       &   && \dTo \\
\left\{\begin{matrix}  Z \\
\downarrow_{\id} \\
Z\end{matrix}\right\}   &&    \rTo    &\left\{\begin{matrix}  X\\
\downarrow_{\id} \\
X\end{matrix}\right\} 
.\end{diagram}
\]
We obtain a pull-back diagram in $\cC$, 
\[
\begin{diagram}
\PB(\pD[m] ,  Z'\times_X Z, \D[m], Z) \raisebox{-0.7cm}{\SEpbk}   &&
\rTo & \PB(\pD[m], Z', \D[m], X)   \\ \\
 \dTo       &   && \dTo \\
Z_m= \PB(\pD[m], Z, \D[m], Z)  &&    \rTo    &  \PB(\pD[m], X, \D[m], X)=X_m 
.\end{diagram}
\]
When $m<n$, notice that
\begin{equation}\label{eq:zz'}
Z'_m \thra \pb(\pD[m], Z', \D[m], X), \quad Z_m \cong Z_m, \quad X_m
\cong X_m;
\end{equation}
then using Lemma \ref{lemma:fp} (in the case $L \cong A$ and $M\cong B $), we conclude that $Z_m \times_{X_m} Z_m
\thra \pb(\pD[m] ,  Z'\times_X Z, \D[m], Z)$ is a cover in $\cC$. When
$m= n$, we only have to change the $\thra$ in \eqref{eq:zz'} to
$\cong$ to obtain  $Z_m \times_{X_m} Z_m
\cong \pb(\pD[m] ,  Z'\times_X Z, \D[m], Z)$.  
\end{proof}

\begin{lemma}\label{lemma:fp}
Given a pull-back diagram in $\cC$,
\[
\xymatrix{B\times_A C \ar[r] \ar[d] & C\ar[d] \\
B \ar[r] & A,}
\]
covers $L \to A$, $M \to B$, $N \to
L\times_{A}C$, and a morphism $M \to L $, then the natural map
$M\times_L N \to B\times_A C$ is a cover. Moreover when $M\to B$ and
$N\to L\times_A C$ are isomorphisms, $M\times_L N \to B\times_A C$ is an isomorphism. 
\end{lemma} 
\begin{proof}
We form the following pull-back diagram (where $\bigcirc$ denotes
unimportant  pull-backs),
\[
\xymatrix@C=.3cm@R=.2cm{M \times_L N \ar@{->>}[dd] \ar[rrr] & & & N \ar@{->>}[dd] & \\
& &  B\times_A C \ar[rr] \ar[dd]& & C \ar[dd] \\
\bigcirc \ar[r] \ar@{->>}[urr] \ar[dd] & \bigcirc \ar[rr] \ar[ur] \ar[dd] & &
L\times_{A} C \ar[dd] \ar[ur] & \\
& &  B\ar[rr] & & A \\
M \ar@{.>}[r] \ar@{->>}[urr] \ar@/_1pc/[rrr] & B\times_A L \ar[rr] \ar[ur] & & L
\ar@{->>}[ur] &
}
\]
Since $M$ maps to both $B$ and $L$, there exists a morphism $M\to
B\times_A L$, fit into the diagram above. Since $L\thra A$, all the
objects in the diagram are representable in $\cC$. Then the natural
map $M\times_L N \to B\times_A C$ as a composition of covers is a
cover itself. The statement on isomorphisms may be proven similarly.
\end{proof}

\subsection{Pull-back,  generalized morphisms and various Morita equivalences }

Let us first make the following
observation:  when $n=1$ and $\cC$ is the category of Banach
manifolds, \equivalences\ of $n$-groupoid objects give the concept of
equivalence (or pull-back) of Lie groupoids.  We explicitly study the
case when $n=2$: Let $X$ be a $2$-groupoid object in $\cC$ and let $Z_1\rra Z_0$ be in $\cC$ with structure maps as in \eqref{eq:face-degen} up to
the level $m\leq 1$, and $f_m: Z_m \to X_m$ preserving the
structure maps $d^{m}_k$ and $s^{m-1}_k$ for $m\leq 1$. Then
$\hom(
\pD[m] , Z)$ still makes sense for $m\leq 1$. We further suppose
$f_0: Z_0 \twoheadrightarrow X_0$ (hence $Z_0\times Z_0
\times_{X_0\times X_0} X_1 \in \cC$) and $Z_1\thra Z_0\times
Z_0 \times_{X_0\times X_0} X_1$. That is to say  that the induced
map from $Z_m$ to the pull-back $\pb(
\pD[m] , Z, \D[m], X)$
is a cover for $m= 0, 1$. Then we form\footnote{Strictly speaking, $Z$ is not a simplicial object, but $\hom(
\partial \Delta[2] , Z)$ as a fibre product of $Z_1$'s over $Z_0$'s still makes sense.}
\[Z_2= \Pb(\hom(
\partial \Delta[2] , Z)\ra \hom(\partial \Delta[2], X) \la X_2).\]
It is easy to see that the proof of Lemma
\ref{lemma:indct-mfd-equi} still guarantees $Z_2\in \cC$.
Moreover there are $d^2_i: Z_2\to Z_1$ induced by the natural
maps $\hom(\partial \Delta[2] , Z)\to Z_1$; $s^1_i: Z_1 \to
Z_2$ by
\[ s^1_0(h)=(h,h,s^0_0(d^1_0(h)),s^1_0(f_1(h))), \quad
s^1_1(h)=(s^0_0(d^1_1(h)),h,h,s^1_1(f_1(h)));
\]
$m_i^Z: \Lambda(Z)_{3, i}\to Z_2$ via $m^X_i: \L(X)_{3,i} \to X_2$ by for example
\[ m_0^Z(( h_2, h_5, h_3, \bareta_1), (h_4, h_5, h_0, \bareta_2), (h_1, h_3, h_0, \bareta_3))=
( h_2, h_4, h_1, m^X_0(\bareta_1, \bareta_2, \bareta_3)), \] and
similarly for other $m$'s.
\[
\begin{xy}
*\xybox{(0,0);<3mm,0mm>:<0mm,3mm>::
  ,0
  ,{\xylattice{-5}{0}{-4}{0}}}="S"
  ,{(-10,-10)*{\bullet}}, {(-10, -12)*{_1}},
     ,{(0,0)*{\bullet}}, {(0, 2)*{^{0}}}, {(10, -10)*{\bullet}},
     {(10, -12)*{_2}}, {(15, -4)*{\bullet}}, {(17,-5)*{^3}},
     {(-10, -10) \ar@{->}^{h_0} (0,0)},
     { (10, -10) \ar@{->}^{h_1} (-10,-10)},
      { (10, -10) \ar@{->} (0,0)}, {(2, -5)*{^{h_3}}}, 
     {(15, -4)\ar@{->}^{h_2} (10, -10)},
     {(15, -4)\ar@{->}_{h_5} (0, 0)},
     {(15, -4)\ar@{.>} (-10,-10)}, {(-2,-7)*{^{h_4}}},
\end{xy} \]
Then $Z_2
\Rrightarrow Z_1 \rra Z_0$ is a $2$-groupoid object in $(\cC, \cT)$, and we call it the
\emph {pull-back $2$-groupoid} by $f$. Moreover $f: Z\to X$ is \an
\equivalence\ with $f_0$, $f_1$ and the natural map $f_2: Z_2 \to X_2$.

\begin{defi}\label{defi:gen-morp} A \emph {generalized morphism} between two $n$-groupoid objects $X$ and
  $Y$ in $(\cC, \cT)$ consists of a \emph {zig-zag}  of strict maps
$X\stackrel{\sim}{\leftarrow}Z\to Y$, where the map $Z\stackrel{\sim}{\to} X$ is
\an \equivalence. 
\end{defi}

\begin{prop}A composition of generalized morphisms is still a
generalized morphism.
\end{prop}
\begin{proof}This follows from Lemmas \ref{lemma:comp-equi},  \ref{lemma:ngpd-fp} and \ref{lemma:inverse-comp-equi}.
\end{proof}

\begin{defi}\label{defi:m-equi-2gpd} Two $n$-groupoid objects $X$ and
  $Y$ in $(\cC, \cT)$ are
\emph {Morita equivalent} if there is another $n$-groupoid object  $Z$
in $(\cC, \cT)$ and
maps $ X\stackrel{\sim}{\leftarrow} Z
\stackrel{\sim}{\rightarrow} Y$ such that both maps are \equivalences. By Lemmas \ref{lemma:comp-equi},  \ref{lemma:ngpd-fp} and \ref{lemma:inverse-comp-equi}, this definition does give an equivalence
relation. We call it \emph {Morita equivalence} of \ngpd.
\end{defi}

However, Morita equivalent Lie $2$-groupoids correspond to Morita
equivalent SLie groupoids \cite{bz}. Hence to obtain isomorphic
stacky groupoid objects, we need a stricter equivalence relation.

\begin{pdef}\label{defi:1t-m-equi-2gpd}
A strict map of $n$-groupoid objects $f: Z \to X$ is a \emph {$1$-\equivalence}
if it is \an \equivalence\ with $f_0$ an
isomorphism. 
Two $n$-groupoid objects $X$ and $Y$ in $(\cC, \cT)$ are $1$-Morita equivalent if there is an
$n$-groupoid object $Z$ in $(\cC, \cT)$ and maps $ X\stackrel{\sim}{\leftarrow} Z
\stackrel{\sim}{\rightarrow} Y$ such that both maps are
$1$-\equivalences. This gives an equivalence relation between
$n$-groupoid objects, and we call it \emph {$1$-Morita equivalence}.
\end{pdef}
\begin{proof}
It is easy to see that the composition of $1$-\equivalences\ is still
a $1$-\equivalence.  We just have to notice that if both \equivalences\ $f:
Z \to X$ and $f': Z' \to X$ are
$1$-\equivalences, then the natural maps $Z_0\la Z_0\times_{X_0} Z'_0
\ra Z'_0$ are isomorphisms since $Z_0\times_{X_0} Z'_0\cong
Z_0\cong X_0 \cong Z'_0$.
\end{proof}
\begin{remark} \label{rk:1t-m-equi}
For a $1$-\equivalence\ $Z\to X$, since $f_0: Z_0 \cong X_0$, we have
$\hom(\partial \Delta^1, Z) = \hom(\partial \Delta^1, X)$. So the
condition on $f_1$ in Def.\ \ref{defi:1t-m-equi-2gpd} becomes
$f_1: Z_1 \thra X_1$.
\end{remark}

\subsection{$\Cosk^m$,  $\Sk^m$ and finite data description}

Often the conventional way with only finite layers of data to
understand Lie group(oid)s is more conceptual in differential
geometry. For a finite description of
an $n$-groupoid,  we introduce the
functors $\Sk^m$  and $\Cosk^m$ from the category of simplicial objects
in sheaves on $\cC$ to itself \cite[Section 2]{duskin2}. It is easy to describe $\Sk^m$: $\Sk^m(X)_k = X_k$
when $k\leq m$ and $\Sk^m(X)_k$ only has degenerated simplices coming from $X_m$
when $k>m$.  Then $\Cosk^m$ is the right adjoint; that is,
\[ \hom(\Sk^n(Y), X) \cong \hom (Y, \Cosk^n(X)). \]

Presumably, $\Sk^m$ can be easily defined as a functor  from the category of simplicial objects
in  $\cC$ to itself. But $\Cosk^n$ involves taking limit. If $\cC$ does not have all limits, we need to go to the category of sheaves. To use the result of \cite{duskin2} without further complications, we need to introduce the concept of point (see \cite[Section 4]{Rogers-Zhu}).
\begin{defi}
A {\bf point} is a functor $p$ from the category of sheaves on $\cC$ to that of sets, which preserves finite limit sand small colimits. A collection of points $\mathcal{P}$ of $(\cC, \cT)$ is called {\bf jointly conservative}, when a morphism $\phi: F\to G$ in  the category of sheaves on $\cC$ is an isomorphism if and only if $p(\phi): p(F) \to p(G)$ is an isomorphism of sets for all $p\in \mathcal{P}$. 
\end{defi}
It is shown in \cite[Prop.4.2]{Rogers-Zhu} that the category of Banach manifolds with surjective submersions has jointly conservative collection of points. 

\begin{prop} If $X$ is an $n$-groupoid object in $(\cC, \cT)$, which has jointly conservative points,  then
  $\Cosk^{n+1}(X) =X$. 
\end{prop}
\begin{proof} Take a point $p$, since $p$ preserves finite limit,  $p(\Cosk^{n+1}(X))= \Cosk^{n+1}(p(X)) = p(X)$. The last step of equality follows from the set-theoretical version of this identity, which is shown in \cite[Section 2]{duskin2}. By the property of jointly conservativeness, we have  $\Cosk^{n+1}(X) =X$. 
\end{proof}

This tells us that it is possible to describe an $n$-groupoid object with only the first
$n$ layers and some extra data. The idea is to let $X_{n+1} := \L[n+1, j]
(X) $, which is a certain fibre product involving $X_{k\leq n}$; then
we produce $X$ by
\begin{equation} \label{eq:nerve}
 X = \Cosk^{n+1} \Sk^{n+1} ( X_{n+1} \to X_n \to \dots \to X_0).
\end{equation} 
When $n=1$, this is a groupoid
object in $(\cC, \cT)$, as we have demonstrated in the
introduction. Set-theoretically,  these extra data are
worked out  when
$n=1, 2$  in \cite{duskin2}. We hereby work out the case of
$n=2$ in an enriched category $(\cC, \cT)$, where
representibility in $\cC$ needs to be taken care of.

The extra data for a $2$-groupoid object are
associative
``$3$-multiplications''. 
Following the notion of
simplicial objects, we call $d^m_i$ and $s^m_j$ the face and
degeneracy maps between $X_i$'s, for $i=0,1,2$; they still satisfy the coherence
condition in \eqref{eq:face-degen}. To simplify the notation and
match it with the definition of groupoids, we use the notation
$\bt$ for $d^1_0$, $\bs$ for $d^1_1$ and $e$ for $s^0_0$. Then we
can safely omit the upper indices for $d^2_i$'s and $s^1_i$'s.
Actually we will omit the upper indices whenever it does not cause
confusion. Similarly to the horn spaces $\hom(\L[m,j], X)$, given
only these three layers, we define $\Lambda(X)_{m, j}$ to be the
space of $m$ elements in $X_{m-1}$ glued along elements in
$X_{m-2}$ to a horn shape without the $j$-th face. \vspace{.6cm}

\centerline{\epsfig{file=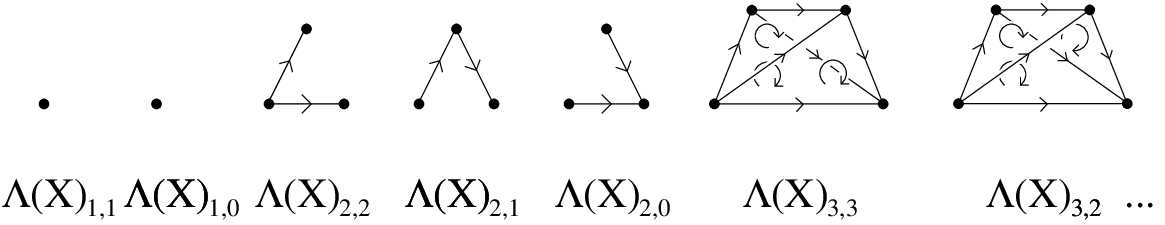,height=1.9cm}} \vspace{.6cm}
\noindent Here one imagines each $j$-dimensional face as an
element in $X_j$. For example,
\[
\begin{split}
\L(X)_{2,2}=X_1\times_{\bs, X_0, \bs}X_1, \; \Lambda(X)_{2, 1} &=
X_1 \times_{\bt, X_0, \bs} X_1, \;
\L(X)_{2,0}=X_1\times_{\bt, X_0, \bt} X_1, \\
\dots, \Lambda(X)_{3,0} &= (X_2\times_{d_2, X_1, d_1}
X_2)\times_{d_1\times d_2, \L(X)_{2,0}, d_1\times d_2} X_2.
\end{split}
\]

We remark that items \eqref{itm:st} and \eqref{itm:dd} in
the proposition-definition below imply that 
the $\L(X)_{2,j}$'s and $\L(X)_{3,j}$'s  are representable in $\cC$. Then with this
condition we can define \emph {$3$-multiplications} as morphisms
$m_i:\Lambda(X)_{3, i}   \to X_2$, $ i=0,\dots,3$. With
$3$-multiplications, there are natural maps between $\L(X)_{3,j}$'s.
For example, \[\L(X)_{3,0} \to \L(X)_{3,1}, \quad \text{by}\;
(\eta_1, \eta_2, \eta_3) \to (m_0(\eta_1, \eta_2, \eta_3), \eta_2,
\eta_3). \] It is reasonable to ask them to be isomorphisms. In
fact, set theoretically, this simply says that the following four
equations are equivalent to each other:
\[
\begin{split}
\eta_0= m_0 (\eta_1, \eta_2, \eta_3), \quad \eta_1 =m_1(\eta_0,
\eta_2, \eta_3), \\ \eta_2=m_2(\eta_0, \eta_1, \eta_3), \quad
\eta_3=m_3(\eta_0, \eta_1, \eta_2).
\end{split}
\]

\begin{pdef}\label{def:finite-2gpd}
A $2$-groupoid object in $(\cC, \cT)$ can be also described by three layers $X_2
\Rrightarrow X_1 \rra X_0 $ of objects in $\cC$ and the following data:
\begin{enumerate}
\item  the face and degeneracy maps $d^n_i$ and $s^n_i$ satisfying
\eqref{eq:face-degen} for $n=1,2$ as explained above, such that
\begin{enumerate}
\item\label{itm:st} [$1$-Kan] $\bt$ and $\bs$ are covers;
\item\label{itm:dd} [$2$-Kan] $d_0\times d_2: X_2\to \L(X)_{2,1}=X_1\times_{\bt, X_0, \bs} X_1$,
$d_0\times d_1: X_2 \to \L(X)_{2,2}=X_1\times_{\bs, X_0, \bs}X_1$,
and $d_1\times d_2: X_2\to \L(X)_{2,0}=X_1\times_{\bt, X_0, \bt}
X_1$ are covers.
\end{enumerate}
\item morphisms ($3$-multiplications),
\[
m_i:  \Lambda(X)_{3, i}   \to X_2, \quad i=0,\dots,3.
\]such that
\begin{enumerate}
\item\label{itm:m-iso} the induced morphisms (by $m_j$ as above) $\L(X)_{3,i}\to \L(X)_{3,j}$ are all
isomorphisms;
\item the $m_i$'s are compatible with the face and degeneracy maps:
\begin{equation}\label{coco}
\hskip - 4 em
\begin{split}
\eta=m_1(\eta, s_0\circ d_1(\eta), s_0\circ d_2(\eta)) &\ \big( \text{which is equivalent to} \; \eta=m_0(\eta, s_0 \circ d_1(\eta), s_0 \circ d_2 (\eta)) \big),\\
\eta= m_2(s_0 \circ d_0 (\eta), \eta, s_1 \circ d_2(\eta)) &\ \big( \text{which is equivalent to} \; \eta= m_1(s_0 \circ d_0 (\eta), \eta, s_1 \circ d_2(\eta)) \big),\\
\eta= m_3 (s_1 \circ d_0 (\eta), s_1 \circ d_1(\eta), \eta) &\
\big( \text{which is equivalent to}\;  \eta= m_2 (s_1 \circ d_0
(\eta), s_1 \circ d_1(\eta), \eta) \big).
\end{split}
\end{equation}
\item the $m_i$'s are associative, that is, for a $4$-simplex $\eta_{0 1 2 3 4}$,
\begin{equation}\label{pic:5-gon}
\begin{xy}
*\xybox{(0,0);<3mm,0mm>:<0mm,3mm>::
  ,0
  ,{\xylattice{-5}{0}{-4}{0}}}="S"
  ,{(-10,-7)*{\bullet}}, {(-10, -9)*{_1}},
     ,
{(0,9)*{\bullet}},  {(0, 11)*{^{0}}},
{(0,3)*{\bullet}}, {(0, 0)*{^{4}}}, {(10, -7)*{\bullet}},
     {(10, -9)*{_{2}}}, {(15, -1)*{\bullet}}, {(17,-2)*{^3}},
     {(-10, -7) \ar@{->} (0,9)}, 
{(-10, -7) \ar@{<-} (0,3)},
     { (10, -7) \ar@{->} (-10,-7)},
 { (10, -7) \ar@{->} (0,9)},
      { (10, -7) \ar@{<-} (0,3)},
     {(15, -1)\ar@{->} (10, -7)},
{(15, -1)\ar@{->}(0,9)},
     {(15, -1)\ar@{<-}(0,3)},
     {(15, -1)\ar@{->}(-10,-7)},
     {(0,3)\ar@{->} (0,9)}
\end{xy} 
\end{equation}
if we are given faces $\eta_{0 i 4}$ and $\eta_{0 ij}$ in $X_2$,
where $i, j \in \{ 1, 2, 3\}$, then the following two methods to
determine the face $\eta_{1 2 3}$ give the same element in $X_2$:
\begin{enumerate} 
\item $\eta_{1 2 3}= m_0( \eta_{0 23},\eta_{01 3}, \eta_{0 1 2})$;
\item we first obtain $\eta_{i j 4}$'s using the $m_i$'s on the $\eta_{0 i 4}$'s; then we have
\[\eta_{123}= m_3( \eta_{2 3 4}, \eta_{1 3 4}, \eta_{1 2 4}).\]
\end{enumerate}
\end{enumerate}
\end{enumerate}
\end{pdef}
\begin{remark} 
  Set-theoretically, this definition is that of \cite{duskin}. In fact, it is
  enough to use one of the four multiplications $m_j$ as therein,
  since one determines the others by item \ref{itm:m-iso}. However, we
  use all the four multiplications here and later on in the proof to
  make it geometrically more direct.  Here we see that this idea
  also applies well to, and even brings convenience to, other categories. 

  For example, in the case of a Lie $2$-groupoid, i.e.\ when $\cC$ is the
  category of Banach manifolds with surjective submersions as covers,
  although the surjectivity of the maps in the $2$-Kan condition
  \eqref{itm:dd} insures the existence of the usual ($2$-) multiplication
  $m: X_1\times_{\bt, X_0, \bs} X_1 $ and inverse $i: X_1 \to X_1$
  as explained in the introduction, these two maps are not necessarily continuous, or
  smooth, and $m$ is not necessarily associative on the nose. For example, the Lie $2$-groupoids
  coming from integrating Lie algebroids have two models
  \cite{z:lie2}: the finite-dimensional one does not have a
  continuous $2$-multiplication and the infinite-dimensional one has a
  smooth multiplication which does not satisfy associativity on the
  nose.

On the other hand, only having the usual $2$-multiplication $m$ and
inverse map $i$, it is not guaranteed that the maps in the $2$-Kan
condition \eqref{itm:dd} are submersions even when $m$ and $i$ are
smooth. But being submersions is in turn very important to prove
that $X_{n\geq 3}$ are smooth manifolds. Hence in the
differential category, we cannot replace the $2$-Kan condition by
the usual $2$-multiplication and inverse.
\end{remark}

\paragraph{The nerve of $X_2
\Rrightarrow X_1 \Rightarrow X_0$} To show that what we defined
just now is the same as Def.\ \ref{defngroupoids}, we form the \emph
{nerve} of a $2$-groupoid $X_2 \Rrightarrow X_1 \Rightarrow X_0$
in Prop-Def.\ \ref{def:finite-2gpd}. We first define
\[X_3=\{ (\eta_0, \eta_1, \eta_2, \eta_3)\mid \eta_0=m_0(\eta_1, \eta_2, \eta_3),
(\eta_1, \eta_2, \eta_3)\in \Lambda(X)_{3,0}\}. \] Then $X_3 \cong
\Lambda(X)_{3,0}$ is representable in $\cC$. Moreover, we have the obvious
face and degeneracy maps between $X_3$ and $X_2$,
\[
\begin{split}
& d^3_i( \eta_0, \eta_1, \eta_2, \eta_3)=\eta_i, \; i=0,\dots,3 \\
& s^2_0(\eta)=(\eta, \eta, s_0 \circ d_1 (\eta), s_0 \circ d_2 (\eta) ), \\
& s^2_1(\eta)=(s_0\circ d_0(\eta), \eta, \eta, s_1\circ d_2(\eta)), \\
& s^2_2(\eta)=(s_1\circ d_0(\eta), s_1\circ d_1(\eta), \eta,
\eta).
\end{split}
\]
The coherency \eqref{coco} insures that $s^2_i(\eta)\in X_3$. It
is also not hard to see that these maps together with the $d^{\leq
2}_i$'s and $s^{\leq 1}_i$'s satisfy \eqref{eq:face-degen} for
$n\leq 3$.

Then the nerve can be easily described by \eqref{eq:nerve}.
More concretely, $X_m$ is made up of those $m$-simplices whose
$2$-faces are elements of $X_2$ and such that each set of four
$2$-faces gluing together as a $3$-simplex is an element of $X_3$.
That is,
\[ X_m = \{ f\in \hom_2( \sk_2(\Delta_m), X_2)\mid f\circ (d_0\times d_1 \times d_2 \times d_3)(\sk_3(\Delta_m))\subset X_3\} ,\]
where $\hom_2$ denotes the homomorphisms restricted to the $0,1,2$
level and $X_2$ is understood as the tower $X_2 \Rrightarrow
X_1\Rightarrow X_0$ with all degeneracy and face maps. Then there
are obvious face and degeneracy maps which naturally satisfy
\eqref{eq:face-degen}.

However what is nontrivial is that the associativity of the $m_i$'s assures that $X_m$
is representable in $\cC$. We prove this by an inductive argument. Let
$S_j[m]$ be the the contractible simplicial set
whose sub-faces all contain the vertex $j$ and whose only
non-degenerate faces are of  dimensions $0$, $1$ and $2$. Then similarly
to \cite[Lemma 2.4]{henriques}, we now show that $\hom_2(S_j[m], X_2) $
is representable in $\cC$. Since $S_j[m]$ is constructed by adding
$0,1,2$-dimensional faces, it is formed by the procedure
\[
\xymatrix{ S' \ar[r] & S \\
           \L[m , j] \ar[u] \ar[r] & \Delta[m]\ar[u]}
\]
with $m\leq 2$. The dual pull-back diagram shows that
$\hom_2(S_j[m], X_2)$ is representable by induction
\[
\xymatrix{ \hom_2(S', X_2)\ar[d] &  \hom_2(S, X_2) \ar[l] \ar[d] \\
           \hom_2(\L[m,j] , X_2) & \hom_2(\D[m], X_2) ,  \ar[l] }
\]
since $\hom_2(\D[m], X_2) \to \hom_2(\L[m,j], X_2) $ are
covers by items \ref{itm:st} and \ref{itm:dd}
in the Prop-Def \ref{def:finite-2gpd}.

Next we use induction to show that $X_m = \hom_2( S_0[m], X_2)$.
Similarly we will have $X_m = \hom_2( S_j[m], X_2)$.  It is clear
that $\tilde{f}\in X_m$ restricts to $\tilde{f}|_{S_0[m]} \in
\hom_2(S_0[m], X_2)$. We only have to show that $f\in
\hom_2(S_0[m], X_2)$ extends uniquely to $\tilde{f}\in X_m$. It is
certainly true for $n=0, 1, 2, 3$ just by definition. Suppose
$X_{m-1}= \hom_2( S_0[m-1], X_2)$. Then to get $f\in
\hom_2(S_0[m], X_2)$ from $f'\in \hom_2( S_0[m-1], X_2)$, we add a
new point $m$ and $(m-1)$ new faces $(0, i, m)$, $i\in \{1, 2,
\dots, m-1\}$ and dye them red\footnote{More precisely, they are
the image of these under the map $f$.}. Using $3$-multiplication
$m_0$, we can determine face $(i, j, m)$ by $(0, i, m)$, $(0, j,
m)$ and $(0, i,j)$ and dye these newly decided faces blue.  Now we
want to see that each of the four faces attached together are in $X_3$;
then $f$ is extended to $\tilde{f} \in X_m$. We consider various
cases:
\begin{enumerate}
\item if none of the four faces contains the vertex $m$, then by
the induction condition, they are in $X_3$.
\item if one of the four faces contains $m$, then
there are three faces containing  $m$; we again have two
sub-cases:
\begin{enumerate}
\item if those three contain only one blue face of the form
$(i, j, m)$, $i, j \in \{ 1, \dots, (m-1)\}$, then the four faces
must contain three red faces and one blue face. According to our
construction, these four faces are in $X_3$;
\item if those three contain more than one blue face, then they must contain exactly three blue faces.  Then according to
associativity (inside the $5$-gon $(0, i, j, k, m)$), these four
faces are also in $X_3$.
\end{enumerate}
\end{enumerate}
Now we finish the induction, hence $X_m$ is representable in $\cC$ and it is
determined by the first three layers. Similarly we can prove $\hom(\L[m,j],
X)=\hom_2(S_0[m], X_2)$. Hence $\hom(\L[m,j], X) =X_m$, and we finish the proof of the
Prop-Def.\ \ref{def:finite-2gpd}, which is summarized in the following
two lemmas:

\begin{lemma}The nerve $X$ of a $2$-groupoid object $X_2\Rrightarrow X_1
  \Rightarrow X_0$ in $(\cC, \cT)$ as in
Prop-Def.\ \ref{def:finite-2gpd} is a $2$-groupoid object in $(\cC, \cT)$ as in Def.\
\ref{defngroupoids}.
\end{lemma}

\begin{lemma}
The first three layers of a $2$-groupoid object in $(\cC, \cT)$ as in Def.\
\ref{defngroupoids} is a $2$-groupoid object in $(\cC, \cT)$ as in Prop-Def.\
\ref{def:finite-2gpd}.
\end{lemma}
\begin{proof} The proof is more complicated and similar to the
case of $1$-groupoids in the introduction. Here we point out that
the $3$-multiplications $m_j$ are given by $\Kan!(3, j)$ and the
associativity is given by $\Kan!(3, 0)$ and $\Kan(4, 0)$.
\end{proof}

\section{Stacky groupoids in various categories} \label{sect:sgpd}
Given a category $\cC$ with a singleton Grothendieck pretopology $\cT$ (not
necessarily satisfying Assumptions \ref{assump:1}), we can
develop the theory of stacks \cite{SGA4}. The Yoneda lemma also holds here;
namely we can embed $\cC$ into the $2$-category of stacks built upon
$(\cC, \cT)$. We call such stacks \emph {representable stacks}. Moreover, weaker
than this, a kind of nice
stack, which we call a \emph {presentable stack},  corresponds to the
groupoid objects in $\cC$. For this one needs another singleton Grothendieck
pretopology $\cT'$.

The theory of presentable stacks in
$(\cC, \cT, \cT')$ (see Table \ref{table:2}) has been developed over the past few decades in the algebraic
category, where they are known as Delign--Munford (DM) stacks and
Artin stacks in the \'etale and general cases respectively (see for example \cite{v1} for a good
summary), and recently in the differential
category by \cite{bx1,
  metzler, pronk} and \cite{m-orbi} (in the context of orbifolds) and the topological
category by \cite{gep-hen, metzler, noohi:top}. We
refer the reader to these references for these concepts and only
sketch the idea here. 

First, to distinguish, we call a cover in  topology $\cT'$ a \emph {projection}. We call a morphism $f: \cX \to \cY$ between stacks in
$(\cC, \cT)$ a \emph
{representable projection} if for every map $U\to \cY$ for $U\in
\cC$, the pull-back map $\cX \times_U \cY \to U$ is a projection in
$\cC$ (this implies that $\cX \times_U \cY$ is representable in
$\cC$). A morphism $f: \cX \to \cY$ between stacks in $(\cC, \cT)$ is an
\emph {epimorphism} if for every $U\to \cY$ with $U\in \cC$, there
exists a cover $V\to U$ in $T$ fit in the following $2$-commutative diagram 
\[ \xymatrix{ V \ar[r]^{\exists} \ar[d] & U \ar[d]_\forall \\
\cX \ar[r]^{f} & \cY.}\]
Then a \emph {presentable stack}\footnote{This is a
slightly different set-up to that in the usual references, but it says exactly
the same thing by \cite[Lemma 2.13]{bx1}. } in $(\cC, \cT, \cT')$ is a stack $\cX$
which possesses a \emph {chart} $X \in \cC$ such that $X \to \cX$ is a
representable projection and an epimorphism w.r.t. $\cT$.  

To define the $2$-category of groupoids
in $(\cC, \cT, \cT')$, we need to
define first a \emph {surjective projection} between presentable stacks.
We adopt the definition in \cite[Section 3]{tz}, which is that $f: \cX \to
\cY$ is a \emph {projection} if $X\times_{\cY} Y
\to Y$ is a projection where $X$ and $Y$ are charts of $\cX$ and
$\cY$ respectively. $f$ is further a \emph {surjective projection} if it
is an epimorphism of stacks. If $f: \cX \to Y$ is a surjective
projection from a presentable stack to an object in $\cC$, then the
fibre product $\cX
\times_Y \cZ$ for any map $\cZ \to Y$ is again a presentable
stack. Then a \emph {groupoid object}
in $(\cC, \cT, \cT')$ is as we imagine: $G:=G_1 \rra G_0$ with $G_i \in
\cC$, all the groupoid structure maps morphisms in $\cC$, and source
and target maps surjective projections. One subtle point
is that for a principal $G$ bundle $X$ over $S$,
the map $X\to S$ has to be a surjective projection.  

The upshot of this
theory is that the $2$-category of presentable stacks is equivalent to
the $2$-category of groupoids\footnote{In the algebraic category, usually we
need more conditions for such a groupoid to present an algebraic
stack. For example, $(\bt, \bs): G_1 \to G_0 \times G_0$ is separated
and quasi-compact \cite[Prop. 4.3.1]{lmb}. But in the differential and
topological categories, we do not require extra conditions. See Table \ref{table:2}.} in $(\cC, \cT, \cT')$.  This implies that
a presentable stack is presented by a groupoid object (which may not be
unique), and a morphism between presentable stacks is presented by an
H.S. bibundle. There is also a correspondence on the level of
$2$-morphisms. 

Moreover,  we have
\begin{lemma} \label{lemma:fur-ass}
If $\cT'$ satisfies Assumptions \ref{assump:1} with terminal object $*'$
and any map $X\to *'$ is an epimorphism w.r.t. $\cT$ for all $X\in \cC$, then surjective projections serve as covers of a
certain singleton Grothendieck pretopology $\cT''$ which satisfies Assumptions
\ref{assump:1}.
\end{lemma}

\begin{proof} The only thing which is not obvious is to check that if
$Z\to X$ is an epimorphism in $(\cC, \cT )$ and $Y\to X$ is a morphism in $\cC$, then the
pull-back $Y\times_X Z \to Y$ is still an epimorphism in $(\cC, \cT) $, for $X, Y, Z \in \cC$.
For any $U\to Y$, we have a composed morphism $U\to X$. Since $Z\to X$
is an epimorphism,  there
exists a cover $V\to U$ in $T$, such that the rectangle diagram in the
diagram below
commutes
\[
\xymatrix{ V \ar[r] \ar@/_3pc/[dd] \ar@{.>}[d]&  U \ar[d] \\
Z\times_X Y \ar[r] \ar[d] & Y \ar[d] \\
Z \ar[r] & X .
}
\]
Hence there exists a morphism $V\to Z\times_X Y$  such that the up-level
small square commutes; that is, $Z\times_X Y \to Y$ is an epimorphism.
\end{proof}

Therefore by definition, we have
\begin{cor}
The groupoid objects and H.S. morphisms in $(\cC, \cT, \cT')$ are exactly
the same as the (\/$1$-) groupoid objects and H.S. morphisms in $(\cC, \cT'')$. 
\end{cor}

 We list possible $(\cC, \cT, \cT')$ and $\cT''$ with their theory of
presentable stacks in Table \ref{table:2}.  
\begin{table} [h!b!p!]
\caption{Example of categories for a theory of presentable stacks} \label{table:2}
\centering
\begin{tabular}{| p{2.5cm} | p{2cm} | p{3cm} |  p{5cm} |} 
\hline
Theory of & $(\cC, \cT, \cT')$  &  covers of $\cT''$  & presented by \\
\hline \hline
Differentiable stacks & $(\cC_1, \cT_1, \cT'_1)$ & same as $\cT'_1$  & Lie groupoids \\
\hline
Topological stacks & $(\cC_2, \cT_2, \cT'_2)$  & surjective maps  with local 
sections  & topological groupoids\\
\hline
Artin stacks &$(\cC_3, \cT_3, \cT'_3)$ &
same as $\cT'_3$ &  groupoid schemes with extra conditions  \\
\hline
\end{tabular}
\end{table}

{\bf From now on, we restrict ourselves to only the first two situations
  described in Table \ref{table:2}; that is, when we mention  $(\cC, \cT,
  \cT')$ and $\cT''$, it is either  $(\cC_1, \cT_1, \cT'_1)$ and
  $\cT''_1$ or  $(\cC_2, \cT_2, \cT'_2)$ and $\cT''_2$. }

\begin{remark}
The definition of a groupoid object in $(\cC, \cT, \cT')$ is the same as a
groupoid object in $(\cC, \cT')$ even though we have to use $\cT$ to
define epimorphisms. For example,  Lie groupoids are the groupoid objects in $(\cC_1, \cT_1,
\cT'_1)$ and also the groupoid objects in $(\cC_1, \cT_1)$, since both require the
source and target to be surjective
submersions. Topological groupoids are the groupoid objects in
$(\cC_2, \cT_2, \cT'_2)$ requiring source and target to be surjective maps with local
sections. But with the identity section, the conditions for the source
and target naturally hold. Hence
topological groupoids are also the groupoid objects in $(\cC_2,
\cT'_2)$.  However the definition of H.S. morphisms in $(\cC,
\cT, \cT')$ and $(\cC, \cT')$ is not necessarily the same. Hence when
the condition in Lemma \ref{lemma:fur-ass} is satisfied, the definition
of $n$-groupoid in $(\cC, \cT'')$ and $(\cC, \cT')$ is not necessarily the same.  
\end{remark}

\begin{defi} [stacky groupoid]\label{def:sliegpd} A stacky groupoid
  object in $(\cC, \cT, \cT')$ 
over an object  $M\in \cC$ consists of the following data:
\begin{enumerate}
\item a presentable stack $\cG$;
\item (source and target) maps $\bar{\bs},
\bar{\bt}: \cG \to M$ which are surjective projections;
\item (multiplication) a map $m: \cG\times_{\bbs, M, \bbt} \cG \to
\cG$, satisfying the following properties:
\begin{enumerate}
  \item\label{itm:m1} $\bbt \circ m=\bbt\circ \pr_1$, $\bbs \circ m=\bbs\circ
  \pr_2$, where $\pr_i: \cG \times_{\bbs,M, \bbt} \cG \to \cG$ is the
  $i$-th projection map $\cG\times_{\bbs, M, \bbt} \cG \to
\cG$;
  \item\label{itm:m-a} associativity up to a $2$-morphism; i.e., there is a $2$-morphism $a$ between
maps $m\circ (m \times \id)$ and $m\circ(\id\times m)$;
  \item\label{itm:a-higher} the $2$-morphism $a$ satisfies a higher coherence described as follows:
let the $2$-morphisms on the each face of the
     cube be $a_i$\footnote{All the $a_i$'s are generated by $a$, except that $a_4$ is $\id$.} arranged in the following way:
front face (the one with the most $\cG$'s) $a_1$,  back $a_5$; up
$a_4$, down $a_2$; left $a_6$, right $a_3$:
$$
\xymatrix@=5pt{
     & & \cG\mathop\times\limits_{M}\cG\mathop\times\limits_{M}\cG \ar[dr]^{m\times \id} \ar[ddd]^{\id\times m}& \\
  \cG\mathop\times\limits_{M}\cG\mathop\times\limits_{M}\cG\mathop\times\limits_{M}\cG \ar[urr]^{\id\times \id\times m} \ar[dr]^{m\times \id\times \id} \ar[ddd]_{\id\times m\times \id} & & & \cG\mathop\times\limits_{M}\cG\ar[ddd]^{m} \\
     & \cG\mathop\times\limits_{M}\cG\mathop\times\limits_{M}\cG\ar[urr]^{\id\times m} \ar[ddd]^{m\times \id} & & & \\
     & & \cG\mathop\times\limits_{M}\cG \ar[dr]^{m} & \\
  \cG\mathop\times\limits_{M}\cG\mathop\times\limits_{M}\cG \ar[urr]^{\id\times m} \ar[dr]^{m\times \id} & & & \cG .\\
     & \cG\mathop\times\limits_{M}\cG \ar[urr]^{m} & & }
$$
We require \[ (a_6\times \id )\circ(\id\times a_2)\circ(a_1\times
\id)=  (\id\times a_5) \circ (a_4 \times \id) \circ (\id \times a_3).
\]
\end{enumerate}

\item  (identity section) a morphism  $\bar{e}$: $M\to \cG$
such that
\begin{enumerate}
\item\label{itm:e-b}  the identities
\[
m\circ ((\bar{e}\circ \bbt)\times \id)\xRightarrow{b_l}\id, \;\;m\circ (\id\times
(\bar{e}\circ\bbs) )\xRightarrow{b_r}\id,\] hold\footnote{In particular, by
combining with the surjectivity of $\bbs$ and $\bbt$, one has
$\bbs \circ \bar{e}= \id$, $\bbt \circ \bar{e}= \id$ on $M$. In fact
if $x=\bbt(g)$, then $\bar{e}(x)\cdot g \sim g$ and $\bbt \circ m =
\bbt \circ \pr_1$ imply that $\bbt(\bar{e}(x))= \bbt (g) = x$. } up
to $2$-morphisms $b_l$ and $b_r$. Or equivalently there are two
$2$-morphisms
\begin{alignat*}{2}
  m \circ (\id \times \bar{e}) & \overset{b_r}{\to} \pr_1 : \cG \times_{\bbs, M} M \to \cG,&
 \quad
 m \circ (\bar{e} \times \id) & \overset{b_l}{\to} \pr_2 : M \times_{M, \bbt} \cG \to \cG , \\
 g\bar{e} (y) & \to g & \quad    \bar{e}(x) g & \to g
\end{alignat*}
where $y=\bbs(g)$ and $x=\bbt(g)$.
\item \label{itm:b-on-M} the restriction of $b_r$ and $b_l$ on $m\circ (\bar{e} \times
  \bar{e}) \xRightarrow[b_r]{b_l} \bar{e} $ are the same; 
\item\label{itm:br} the composed $2$-morphism below, with $y=\bbs(g_2)$,
\[
g_1 g_2   \xrightarrow{b^{-1}_r}   (g_1 g_2) \bar{e}(y)
\xrightarrow{a}  g_1(g_2 \bar{e}(y)) \xrightarrow{b_r}  g_1g_2
\] is the identity;\footnote{We can also state this without any reference to
objects. We notice that $\pr_1 \circ (m \times \id) $ and $m \circ
(\pr_1\times \pr_2)$ are the same map from $\cG\times_M \cG\times_M
M $ to $ \cG$, but as the diagram indicates,
\begin{equation}\label{diag:br}
\xymatrix{{\cG\times_M \cG\times_M M} \ar[d]^{ \pr_1\times \pr_2}
\ar@<-1ex>[d]_{\id \times (m\circ (\id \times \bar{e}))}
\ar[rr]^{ m \times \id} & & {\cG \times_M M} \ar[d]^{ m\circ(\id\times \bar{e})} \ar@<-1ex>[d]_{ \pr_1}\\
{\cG \times_M \cG} \ar[rr]^m & & {\cG}, }
\end{equation}
they are related also via a sequence of $2$-morphisms:
\begin{equation}\label{eq:br}
 \pr_1 \circ (m \times \id) \xrightarrow{b_r^{-1}\odot \id}
m\circ(\id\times \bar{e}) \circ (m \times \id) \xrightarrow{a}  m
\circ (\id \times (m\circ (\id \times \bar{e})) ) \xrightarrow{\id
\odot (\id \times b_r)}  m \circ (\pr_1\times \pr_2),
\end{equation}
where $\odot$ denotes conjunction of $2$-morphisms, so that for example
$b_r^{-1}:\pr_1 \to m\circ(\id\times \bar{e})$ is a $2$-morphism,
$\id:m\times \id \to m\times \id$ is a $2$-morphism, and the
conjunction $b_r^{-1} \odot \id$ gives a $2$-morphism between the
composed morphisms $ \pr_1 \circ (m \times \id) \xrightarrow{b_r^{-1}\odot \id}
m\circ(\id\times \bar{e}) \circ (m \times \id) $.
We require that the composed $2$-morphisms be $\id$, that is that
\[ (\id \odot ( \id \times  b_r)) \circ a \circ (b_r^{-1} \odot \id) =\id
,\] where $\circ$ is simply the composition of $2$-morphisms.}
\item\label{itm:bl}similarly with $x=\bbt(g_1)$,
\[ g_1 g_2   \xrightarrow{b^{-1}_l}   \bar{e}(x)(g_1 g_2)
 \xrightarrow{a^{-1}} (\bar{e}(x)g_1)g_2  \xrightarrow{b_l}  g_1g_2
 \] is the identity;
\item\label{itm:bl-br} with $x =\bbs(g_1)=\bbt(g_2)$,
\[
g_1 g_2 \xrightarrow{b_l^{-1}} (g_1\bar{e}(x)) g_2 \xrightarrow{a}
g_1(\bar{e}(x) g_2) \xrightarrow{b_r} g_1 g_2, \]is the identity.
\end{enumerate}

\item (inverse) an isomorphism of stacks
$\bar{i}: \cG \to \cG$ such that, the
following identities
\[ m\circ (\bar{i}\times \id \circ \Delta)\Rightarrow \bar{e}\circ\bbs, \;\;
m\circ (\id\times\bar{i}\circ \Delta)\Rightarrow \bar{e}\circ \bbt,\]
hold  up to $2$-morphisms,
where $\Delta$ is the diagonal map: $\cG\to \cG\times\cG$.
\end{enumerate}
\end{defi}

We are specially interested in the differential category.
\begin{defi}When  $(\cC, \cT, \cT')$ is the differential category $(\cC_1,
  T_1, \cT'_1)$, we call a stacky
groupoid object $\cG \rra M$ a \emph {stacky Lie groupoid} (\emph {SLie groupoid} for short). When $\cG$ is
furthermore an \'etale differentiable stack and the identity $e$ is an
immersion of differentiable stacks, we call it a \emph {Weinstein groupoid}
(\emph {W-groupoid} for short).
\end{defi}

\begin{remark} \label{rk:slgpd}
 This definition
of W-groupoid is different from the one in \cite{tz}:  here we add various higher coherences on
$2$-morphisms which make the definition stricter but still
allow the W-groupoids $\cG(A)$ and $\cH(A)$, which are the
integration objects of the Lie algebroid $A$ constructed in
\cite{tz}. Hence we remove 
\\
\noindent {\scriptsize ``Moreover, restricting to the identity
section, the above $2$-morphisms between maps are the $\id$
$2$-morphisms. Namely, for example, the $2$-morphism $\alpha$ induces
the $\id$ $2$-morphism between the following two maps:\[ m\circ ((m
\circ (\bar{e}\times\bar{e}\circ \delta))\times \bar{e} \circ
\delta)=m\circ(\bar{e}\times(m\circ(\bar{e}\times\bar{e}\circ\delta))\circ\delta),
\]where $\delta$ is the diagonal map: $M\to M\times M$.''}
\\
since it is implied by item \ref{itm:b-on-M} and item \ref{itm:br}.

On
the other hand, we do not add higher coherences for the
$2$-morphisms involving the inverse map. This is because we can always find $c'_r$ and
$c'_l$ that satisfy correct higher coherence conditions, possibly
with a modified inverse map.
 See Section
\ref{sec:inverse}.

With some patience, we can check that the list of
coherences on $2$-morphisms given here generates all the possible
coherences on these $2$-morphisms. In fact, item \ref{itm:br} and item
\ref{itm:bl} are redundant (see \cite[Chapter
VII.1]{maclane:cat-math}). But we list them here since it makes more
convenient for us to use later.  We also notice that the cube
condition \eqref{itm:a-higher} is the same as  the
pentagon condition \[ \left[ ((gh)k)l \to (g(hk))l\to g((hk)l) \to
g (h(kl)) \right]= \left[ ((gh)k)l\to (gh)(kl)\to g(h(kl)) \right]
.\]

\end{remark}

\subsection{Good charts} \label{sec:embedding}
Given a stacky groupoid $\cG \rra M$ in $(\cC, \cT, \cT')$, the identity map $\bar{e}:
M\to \cG$ corresponds to an H.S. morphism from
$M\rra M$ to $G_1\rra G_0$ for some presentation of $\cG$. But it
is not clear whether $M$ embeds into $G_0$. It is not even obvious
whether there is a map $M\to G_0$.  In general, one could ask: if
there is a map from an $M \in \cC$ to a presentable stack
$\cG$, when can one find a chart $G_0$ of $\cG$ such that $M\to
\cG$ lifts to $M\to G_0$, namely when is the H.S. morphism $M\rra
M$ to $G_1\rra G_0$ a strict groupoid morphism? If the stack $\cG$
is \'etale, can we find an \'etale chart $G_0$?  If such $G_0$ exists, we call it a
\emph {good chart} or \emph {good \'etale chart} if it is furthermore
\'etale,  and we call $G_1 \rra G_0$ a \emph {good
  groupoid presentation} for the map $M\to \cG$.  

We show the existence of good (\'etale) charts  in the differential category
$(\cC_1, \cT_1, \cT'_1)$ by the following lemmas.  It turns out that the \'etale
case is easier and when $M\to \cG$ is an immersion we can always
achieve an \'etale chart.

\begin{lemma}\label{lemma:embedding}
For an immersion $\bar{e}: M\to \cG$ from a manifold $M$ to an
\'etale stack $\cG$,  there is an \'etale chart $G_0$ of $\cG$
such that $\bar{e}$ lifts to an embedding $e: M\to G_0$.
\end{lemma}
\begin{proof}
Take an arbitrary \'etale chart $G_0$ of $\cG$. The idea is to
find an ``open neighborhood'' $U$ of $M$ in $\cG$ with the
property that $M$ embeds in $U$ and there is an \'etale
representable map $U\to \cG$. Since $G_0\to \cG$ is an \'etale
chart, in particular epimorphic, $G_0\sqcup U \to \cG$ is an
\'etale representable epimorphism\footnote{Note that being \'etale
implies being submersive.}, that is, a new \'etale chart of $\cG$.
Then the lemma is proven since $M\hookrightarrow G_0\sqcup U$ is
an embedding.

Now we look for such a $U$. Since $\bar{e}:M \to \cG$ is an
immersion, the pull-back $M\times_{\cG}G_0 \to G_0$ is an
immersion and $M\times_{\cG}G_0 \to M$ is an \'etale epimorphism.
We cover $M$ by small enough open charts $V_i$ so that each $V_i$
lifts to an isomorphic open chart $V'_i$ on $M\times_{\cG} G_0$.
Then $V'_i \to G_0$ is an immersion, so locally it is an embedding.
Therefore we can divide $V_i$ into even smaller open charts
$V_{i_j}$ such that $V_{i_j}\cong V'_{i_j}\to G_0$ is an
embedding. Hence we might assume that the $V_i$'s form an open
covering of $M$ such that $\bar{e}$ lifts to embeddings $e_i: V_i
\hookrightarrow G_0$. This appears in the language of
Hilsum--Skandalis (H.S.) bibundles as the diagram on the right:
\[
\xymatrix{V'_i\subset & M\times_{\cG} G_0 \ar[r] \ar[d] & G_0 \ar[d] \\
V_i \subset & M \ar[r] & \cG} \quad
\xymatrix{ M\ar[dd]\ar@<-1ex>[dd]&  &G_1\ar[dd]\ar@<-1ex>[dd]\\
& V'_i\subset M\times_{\cG}G_0\ar[dr]^{J_r}\ar[dl]_{J_l}& &\\
M\supset V_i \ar@<-1ex>[ur]_{\sigma_i}&  &G_0.\\
}
\]
Here $e_i= J_r \circ \sigma_i$. Since the action of $G_1$ on the
H.S. bibundle is free and transitive, there exists a unique
groupoid bisection $g_{ij}$ such that $e_i \cdot g_{ij}=e_j$ on
the overlap $V_i\cap V_j$. Since $G_1 \rra G_0$ is \'etale, the bisection $g_{ij}$
extends uniquely to $\bar{g}_{ij}$ on an open set
$\bar{U}_{ij}\subset G_1$. Moreover, there exist open sets $U_i
\supset e_i(V_i)$ of $G_0$ such that
\[e_i(V_i\cap V_j)\subset \bt_G(\bar{U}_{ij}) =: U_{ji} \subset U_i,
\quad e_j(V_i\cap V_j)\subset \bs_G(\bar{U}_{ij}) =: U_{ij} \subset
U_j. \] Since $e_j \cdot g_{ij}^{-1}=e_i$, which implies that $g_{ji}= g_{ij}^{-1}$, these sets are well-defined.   

\centerline{
\epsfig{file=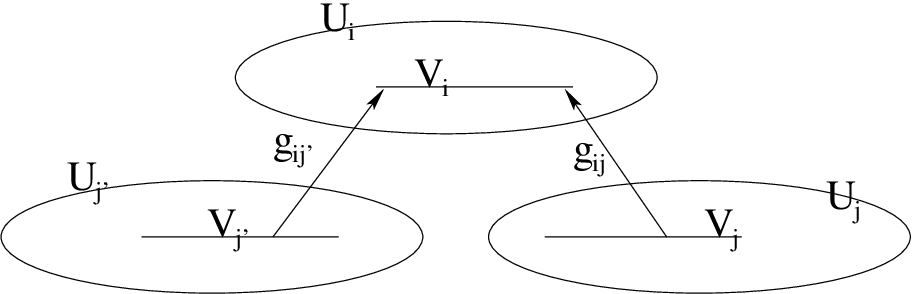,height=2cm}}

Because of uniqueness and because $g_{ij}\cdot g_{jk}=g_{ik}$, we have
$\bar{g}_{ij}\cdot \bar{g}_{jk}=\bar{g}_{ik}$ on the open subsets
$\bar{U}_{ijk}:= \{ (\bar{g}_{ij}, \bar{g}_{jk}, \bar{g}_{ik}):$
$\bar{g}_{ij}\cdot \bar{g}_{jk}$ exists and is in
$\bar{U}_{ik}\}$. Then \[e_i(V_i \cap V_j \cap V_k) \subset
U_{ijk}:= \bt_G (Im(\bar{U}_{ijk} \to \bar{U}_{ij})) \subset
U_{ji}\cap U_{ki} \subset U_i, \] and similarly for $j$ and $k$.
Therefore with these $U$'s we are in the situation of a \emph {germ
of manifolds}  of $M$ defined as below.

A \emph {germ of manifolds at a point} $m$ is a series of manifolds
$U_i$ containing the point $m$ such that each $U_i$ agrees with $U_j$ in a
smaller open set $(m\mathrel{\in)} U_{ji}\subset U_i$ by $x\sim f_{ji}(x)$, with
$f_{ji}: U_i \to U_j$ satisfying the cocycle condition $f_{kj}\circ f_{ji}=f_{ki}$.
A \emph {compatible riemannian metric} of a germ of manifolds
consists of a riemannian metric $g^i$ on each $U_i$ such that two
such riemannian metrics $g^i$ and $g^j$ on $U_i$ and $U_j$ agree
with each other in the sense that $g^i(x)=g^j(f_{ji}(x))$ in a
smaller open set (possibly a subset of $U_{ji}$).  With this, one
can define the exponential map $\exp$ at $m$ using the usual
exponential map of a riemannian manifold, provided the germ is
finite, meaning that there are finitely many manifolds in the germ (which
is true in our case, since $V_i$ intersects finitely many other $V_j$'s).
Then $\exp$ gives a Hausdorff manifold containing $m$.

If a series of locally finite manifolds $U_i$ and morphisms
$f_{ji}$ form a germ of manifolds for every point of a manifold
$M$, we call it \emph {a germ of manifolds of} $M$. Here local
finiteness means that any open set in $M$ is contained in finitely
many $U_i$'s and $M$ has the topology induced by the $U_i$'s, that
is that $M\cap U_i$ is open in $M$.  We can always endow each of these
with a compatible riemannian metric, beginning with any riemannian
metric $g^i$ on $U_i$ and modifying it to the sum
$g'^i(x):=\sum_{k, x\in U_{ki}} g^k(f_{ki}(x))$ (with
$f_{ii}(x)=x$) at each point $x \in U_i$. In this situation, one
can take a tubular neighborhood $U$ of $M$ by the $\exp$ map of the
germ. Then $U$ is a Hausdorff manifold.

Applying the above construction to our situation, we have a
Hausdorff manifold $U \supset M$ with the same dimension as $G_0$.
$U$ is basically glued by small enough open subsets
$\tilde{U}_i=U\cap
U_i$ containing the $V_i$'s along $\tilde{U}_{ij}:=  U\cap U_{ij}$ 
so that the gluing result $U$ is still a Hausdorff manifold.
Therefore $U$ is presented by $\sqcup \tU_{ij} \rra \sqcup \tU_i$,
which maps to $G_1\rra G_0$ via $U_{ij}\cong
\bar{U}_{ij}\hookrightarrow G_1$. So there is a map $\pi: U\to
\cG$. Since the $\tU_{i} \to G_0$ are \'etale maps, by the technical
lemma below, $\pi$ is a representable \'etale map.
\end{proof}

\begin{lemma}\label{lemma:rep-sub}
Given a manifold $X$ and an (\'etale) differentiable stack $\cY$,
a map $f: X\to \cY$ is an (\'etale) representable submersion if
and only if there exists an (\'etale) chart $Y_0$ of $\cY$ such
that the induced local maps $X_i\to Y_0$ are (\'etale)
submersions, where $\{X_i\}$ is an open covering of $X$.
\end{lemma}
\begin{proof}
For any $V\to \cY$, $X_i\times_{\cY} V = X_i \times_{Y_0} Y_0
\times_{\cY} V $ is representable and $X_i\times_{\cY} V \to V$ is
an (\'etale) submersion since $X_i\to Y_0$ and $Y_0\to \cY$ are
representable (\'etale) submersions. Since the $X_i$'s glue together to
$X$, the $X_i\times_{\cY} V$ with the inherited gluing maps glue to a
manifold $X\times_{\cY} V$. Since being an (\'etale) submersion is
a local property, $X\times_{\cY} V \to V$ is an (\'etale)
submersion.
\[
\xymatrix{ X_i \times_{\cY} V \ar[drr] \ar[dd] \\
& \curvearrowright & X_k \times_{\cY} V \ar[dll] \\
X_j \times_{\cY} V } \quad \xymatrix{&\\ \Leftarrow \\ &} \quad \xymatrix{  X_i  \ar[drr] \ar[dd] \\
& \curvearrowright & X_k  \ar[dll] \\
X_j  }
\]
\end{proof}

\begin{remark}\label{rk:loc} If $\bar{e}$ is the identity map of a W-groupoid $\cG \rra M$, then an
open neighborhood of $M$ in $U$ has an induced local groupoid
structure from the stacky groupoid structure \cite[Section~5]{tz}.
\end{remark}

We further prove the same lemma in the non-\'etale case.
\begin{lemma}\label{lemma:embedding-non-etale}
For a morphism $\bar{e}: M\to \cG$ from a manifold $M$ to a
differentiable stack $\cG$,  there is a chart $G_0$ of $\cG$ such
that $\bar{e}$ lifts to an embedding $e: M\to G_0$.
\end{lemma}
\begin{proof} We follow the proof of the \'etale case, but replace ``\'etale map'' with ``submersion''.
We need a $U$ with a representable submersion  to $\cG$ and an
embedding of $M$ into $U$. There are two differences: first, $V_i$
embeds in $V'_i$ instead of being isomorphic to it, and we do not
have an embedding $V'_i \hookrightarrow G_0$; second, since
$G_1\rra G_0$ is not \'etale, the bisection $g_{ij}$ does not
extend uniquely to some $\bar{g}_{ij}$ and we cannot have the
cocycle condition immediately.

The first difference is easy to compensate for: given any morphism $f:
N_1 \to N_2 $, we can always view it as a composition of an
embedding and a submersion as $N_1 \overset{\id\times
f}{\hookrightarrow} N_1\times N_2
\overset{\pr_2}{\twoheadrightarrow} N_2$. In our case, we have the
decomposition $M\times_{\cG}G_0 \hookrightarrow H_0
\twoheadrightarrow G_0$; then we use the pull-back groupoid $H_1:=
G_1\times_{G_0\times G_0} H_0\times H_0$ over $H_0$ to replace
$G$. Thus we obtain an embedding $V'_i \to H_0$ and so an
embedding $V_i \to H_0$. Then since $H_1\rra H_0$ is Morita
equivalent to $G_1\rra G_0$, we just have to replace $G$ by $H$ or
call $H$ our new $G$. It was not possible to do so in the \'etale
case since $H_0$ might not be an \'etale chart of $\cG$.

For the second difference, first of all we could assume $M$ to be
connected to construct such a $U$. Otherwise we take the disjoint
union of such $U$'s for each connected component of $M$.

Then take any $V_i$ and consider all the charts $V_j$
intersecting $V_i$. We choose $\bar{g}_{ij}$ extending $g_{ij}$ on
an open set $\bar{U}_{ij}$. As before we define the open sets
$U_i$, the $U_j$'s and the $U_{ij}$'s. Then for $V_j$ and $V_{j'}$ both
intersecting $V_i$, we choose $\bar{g}_{jj'}$ to be the one
extending (see below) $\bar{g}_{ij}^{-1}\bar{g}_{ij'}$ with
$\bs_G(\bar{g}_{ij}^{-1}\bar{g}_{ij'})$ in the triple intersection
$\bg_{ij'}^{-1}\cdot(\bg_{ij}\cdot U_j ) \cap U_{j'}$, where
multiplication applies when it can. Since the $\bar{g}$'s are local
bisections, $\bar{g} \cdot -$ is an isomorphism. Identifying via
these isomorphisms, we view and denote the above intersection as
$U_{j'ij}$ for simplicity.

Now we clarify in which sense and why the extension always exists.
Let us assume $\dim M =m$, $\dim G_i =n_i$. 
Here we identify each $V_j$ with its embedded image in $G_0$ and require every $V_j$ to be relatively closed in $U_j$.
Then since we are dealing with local charts, we might assume that
both $\bt_G$ and $\bs_G$ of $G_1\rra G_0$ are just projections from
$\R^{n_1}$ to $\R^{n_0}$. A section of $\bs_G$ is a vector valued
function $\R^{n_0}\to \R^{n_1}/\R^{n_0}$, and its being  a
bisection, namely also a section of $\bt_G$,  is an open condition.
That is, we can always perturb a section to get a bisection. Let $U_{j'ij}:=
\bs_G(Im (\bar{U}_{jij'} \to \bar{U}_{ij'})) \subset U_{j'}$ where
$\bar{U}_{jij'}$ and $\bar{U}_{ij'}$ are defined as before in Lemma
\ref{lemma:embedding}.  If we can
extend $\bg_{ij}^{-1}\bg_{ij'}$ and $g_{jj'}$ from $U_{j'ij} \cup
e_{j'}(V_j\cap V_{j'})$ to a bisection $\bg_{jj'}$ such that $\{
\bs_G(\bar{g}_{jj'})\}$ is an open set in $U_{j'}$, then we obtain a
bisection $\bar{g}_{jj'}$ from $U_{jj'}:=
\bg_{jj'}^{-1}(\{\bt_G(\bg_{jj'})\}\cap U_{j})$ to $U_{j'j} :=
\{\bt_G(\bg_{jj'})\}\cap U_{j}$. It is easy to see that
the $U_{jj'}\cong U_{j'j}$ are open in $G_0$ since
$\{\bt_G(\bg_{jj'})\}\cong \{\bs_G(\bg_{jj'})\}$.

Therefore we are done as long as we can extend a smooth function
$f$ from the union of an open submanifold $O$ with a closed
submanifold $V$ of an open set $B\subset \R^{n_0}$ to the whole
$B$. Since $V$ is closed, using its tubular neighborhood and
partition of unity, we can first extend $f$ from $V$ to $B$ as
$\tilde{f}$. Then $f_1= f-\tilde{f}|_{O\cup V}$ is 0 on $V$. We
shrink the open set $O$ a little bit to $O_i$ such that $V \cap O
\subset O_2\subset O_1 \subset O$. Then we always have a smooth
function $p$ on $B$ with $p|_{\bar{O}_2 } =1$ and $p|_{B-O_1} =0$.
Then the extension function $\tilde{f}_1$ is defined by
\[\tilde{f}_1(x)=
  \begin{cases}
    f_1(x) \cdot p(x) & \text{$x\in O$}, \\
    0 & \text{otherwise}.
  \end{cases}
\] It is easy to see that $\tilde{f}_1$ is smooth, and it agrees with $f_1$ on $O_2$ and $V$ because
$V-O_2=V-O_1 \subset B-O_1$ and $p|_{V-O_2}=0$. Hence $\tilde{f} +
\tilde{f}_1$ extends $f|_{O_2\cup V}$. Now we extend the
$\bar{g}_{ij}^{-1}\bar{g}_{ij'}$'s to $\bar{g}_{jj'}$'s; then the
$\bar{g}$'s satisfy the cocycle condition on smaller open sets of
the triple intersections $U_{j'ij}$ by construction.

Then we view $V_i \cup(\bigcup_{j:V_i\cap V_j \neq \emptyset} V_j)$
as one chart. Notice that a connected manifold is path connected.
Also notice that we didn't use any topological property of $V_i$
or $U_i$. This construction will eventually extend to the whole
manifold $M$ and obtain the desired $\bg_{ij}$'s. Therefore we are
again in the situation of a germ of manifolds and we can apply the
proof of Lemma \ref{lemma:embedding} to get the result.
\end{proof}

\subsection{The inverse map}\label{sec:inverse}
In this section, we prove that the axioms involving the inverse
map in the definition of stacky groupoid can be described by the multiplication and the identity.

Let $\cG\rra M$ be a stacky groupoid object in $(\cC, \cT, \cT')$, and $G:=G_1
\underset{\bt_G}{\overset{\bs_G}{ \rra }}G_0$ a good groupoid
presentation of $\cG$ as described in Section
\ref{sec:embedding}. 
So there is a map $e: M\to G_0$ presenting $\bar{e}$.  We look at
the diagram
\[ \cG \times_{M} \cG \overset{m}{\lra}\cG \overset{\be}{\lla}M \]
and its corresponding groupoid picture,
\begin{equation}\label{eq:inv-construct}
\xymatrix{ G_1\times_{\bs_G \bs, M, \bt_G \bt} G_1 \ar[dd]\ar@<-1ex>[dd]&  &G_1\ar[dd]\ar@<-1ex>[dd]& &M\ar[dd]\ar@<-1ex>[dd]\\
& E_m\ar[dr]^{J_r}\ar[dl]_{J_l}& & {E_{\be}} \ar[dr]^{J_r}\ar[dl]_{J_l}& &\\
G_0\times_{\bs, M, \bt} G_0&  &G_0& &M\\
}
\end{equation}
where $E_m$ and $E_{\be} =G_1\times_{\bt_G, G_0, e}M$ 
are bibundles presenting the multiplication $m$ and identity $\be$
of $\cG$ respectively. We can form a left $G\times_{\bs, M, \bt} G$ module
$E_m \times_{G_0} E_{\be} /G$.  Examining the $G$ action on
$E_{\be}$, we see that the geometric quotient,
\begin{equation}\label{eq:bd-inv}
 (E_m \times_{G_0} E_{\be})/G = E_m\times_{J_r, G_0, e} M, \quad
\begin{xy}
*\xybox{(0,0);<3mm,0mm>:<0mm,3mm>::,0
  ,{\xylattice{-1}{0}{1}{0}}}="S",
  {(-10,-1)*{\bullet}}, {(-10, -3)*{_1}},
     {(12,2)*{\bullet}}, {(14, 2)*{^{0}}},
     {(10, -1)*{\bullet}}, {(10, -3)*{_{2}}},
     {(-10, -1) \ar@{->} (12,2)}, 
     { (10, -1) \ar@{->} (-10,-1)}, { (10, -1) \ar@{->} (12,2)} , 
\end{xy}
\end{equation}
is representable in $\cC$ by Lemma \ref{lemma:kan22}; and we see that
the natural map $G_0\times_{\bs, M, \bs} G_0 \xrightarrow{\pr_2} G_0 $ is  a
projection. This space should be pictured as the diagram above from
the viewpoint of $2$-groupoids.   Moreover there is a left
$G_1\times_{\bs_G \bs, M , \bt_G \bt} G_1$ action (which
might not be free and proper). Therefore, we can view it as a left
$G$ module with the left action of the first copy of $G$ and a
right $G^{\op}$ module with the left action of the second copy of
$G$. Here $G^{\op}$ is $G$ with the opposite groupoid structure.

\begin{lemma} \label{lemma:kan22}
The morphism $(\pr_2\circ J_l) \times J_r : E_m \to G_0\times_{\bs, M ,
\bs} G_0$
is a projection. 
\end{lemma}
\begin{proof}
Let $f_1: \cG \times_{\bbs,M, \bbt} \cG \to \cG \times_{\bbs, M, \bbs} \cG$ be given by $(g_1, g_2)\mapsto (g_1 \cdot g_2, g_2)$, i.e.\ $f_1 = m\times \pr_2$; let $f_2: \cG \times_M \cG \to \cG \times_M \cG$ be given by $(g_1, g_2)\mapsto (g_1 \cdot g_2^{-1}, g_2)$. Since we have
\[ (g_1, g_2) \overset{f_1}{\mapsto} (g_1 g_2, g_2)\overset{f_2}\mapsto ((g_1 g_2)g_2^{-1}, g_2))\sim (g_1, g_2), \]
and
\[ (g_1, g_2) \overset{f_2}{\mapsto}(g_1 g_2^{-1}, g_2)\overset{f_1}{\mapsto}((g_1 g_2^{-1})g_2 , g_2))\sim (g_1, g_2), \]
$f_1 \circ f_2$ and $f_2\circ f_1$ are isomorphic to
$\id$ via $2$-morphisms. Therefore $f_1$ is an isomorphism of stacks.
Therefore $E_m \times_{\pr_2 \circ J_l, G_0, \bt_G} G_1$, presenting
$f_1$, is a Morita bibundle from the Lie groupoid $G_1 \times_M
G_1\rightrightarrows G_0 \times_M G_0$ to $G_1\times_M G_1
\rightrightarrows G_0\times_M G_0$. Hence the two moment maps
$J_l$ (of $E_m$) and $J_r \times \bs_G$ are surjective
projections. Moreover $J_r \times \bs_G$ is invariant under  the left
groupoid action of  $G_1 \times_M G_1$, so in particular under  the
action of  the
second copy. Notice that a $G$ invariant projection $X\to Y$ descends to a
projection $X/G \to Y$ if the $G$ action is principal, in both of our two
cases. As a result, the morphism  $(\pr_2\circ J_l) \times J_r : E_m \to
G_0\times_{M } G_0$ is a projection. 

Moreover since the left
groupoid action of  $G_1 \times_M G_1$ is principal on the
bibundle $E_m \times_{\pr_2 \circ J_l, G_0, \bt_G} G_1$, the induced
action of the first copy of $G_1$ on the quotient  $E_m $ is principal.
\end{proof}

\begin{lemma} \label{proof-morita} The bibundle \eqref{eq:bd-inv} is a Morita
 equivalence
 from $G$ to $G^{\op}$ with moment maps $\pr_1\circ J_l$ and $\pr_2
 \circ  J_l$.
\end{lemma}
\begin{proof}
The left action of $G$ is principal followed by the principal action
of $G$ on $E_m$ proven in Lemma
\ref{lemma:kan22}, and the proof of the principality of the  $G^{\op}$ action is similar
(one considers $G_1\times_{G_0}E_m$).
\end{proof}
\begin{remark}
Another fibre product $E_m \times_{\pr_2\circ J_l, G_0, e}M $ is
isomorphic to $ G_1$ trivially via $b_r$. But the morphisms we use
to construct the fibre product are different.
\end{remark}

Notice that using the inverse operation, a $G^{\op}$ module is also
a $G$ module. In other words, the above lemma says that
$E_m\times_{J_r, G_0, e} M$ is a Morita bibundle between $G$ and
$G$ where the right $G$ action is via the left action of the
second copy of $G\times_M G$ composed with the inverse. With
this viewpoint, we have a stronger statement:

\begin{lemma} \label{isom}
As Morita bibundles from $G$ to $G$, $E_m\times_{J_r, G_0, e} M$
and $E_i$ are isomorphic.
\end{lemma}

\begin{proof}
We know from the property of $E_i$ that $g\cdot g^{-1} \sim 1$; that
is,
there is an isomorphism of  H.S. bibundles \[((G_1\times_{\bs_G,G_0,J_l}E_i)\times_{\bt_G \times J_r,
G_0\times_M G_0}E_m) /G_1\times_M G_1 \cong G_0\times_{e\circ \bt,
G_0, \bt_G}G_1 , \] where  $G_0\times_{e\circ \bt, G_0, \bt_G}G_1$ presents the map
$e\circ \bbt: \cG\to M \to \cG$.  We will first show that $E_m
\times_{G_0}M$ also has this property.

Let $(\gamma_3, \eta_1, \eta_0)\in ((G_1\times_{G_0}(E_m\times_{J_r,
G_0, e} M))\times_{G_0\times_M G_0}E_m)$ (see
\eqref{diag:inverse}). 
\begin{equation} \label{diag:inverse}
\begin{xy}
*\xybox{(0,0);<3mm,0mm>:<0mm,3mm>::
  ,0
  ,{\xylattice{-5}{0}{-4}{0}}}="S",
  {(-10,-10)*{\bullet}}, {(-10, -12)*{_2}},
     {(12,-7)*{\bullet}}, {(14, -7)*{^{0}}},
     {(10, -10)*{\bullet}}, {(10, -12)*{_{3}}},  {(7, -9)*{^{\eta_1}}},
     {(-10, -10) \ar@{->} (12,-7)}, 
     { (10, -10) \ar@{->} (-10,-10)}, { (10, -10) \ar@{->} (12,-7)} , 
\end{xy} \quad
\begin{xy}
*\xybox{(0,0);<3mm,0mm>:<0mm,3mm>::
  ,0
  ,{\xylattice{-5}{0}{-4}{0}}}="S",
  {(-10,-10)*{\bullet}}, {(-10, -12)*{_2}},
     {(12,-7)*{\bullet}}, {(14, -7)*{^{0}}},
     {(10, -10)*{\bullet}}, {(10, -12)*{_{3}}},
     {(10, -4)*{\bullet}}, {(12,-4)*{^{1}}},
     {(-10, -10) \ar@{->} (12,-7)}, 
     { (10, -10) \ar@{->} (-10,-10)}, { (10, -10) \ar@{->} (12,-7)} , 
      {(10, -4)\ar@{<-}(-10, -10)},  
      {(10, -4)\ar@{<-}(12, -7)}, {(10, -4) \ar@{<-} (10, -10)}, 
      {(6, -7)*{^{\gamma_3}}}, {(7, -9)*{^{\eta_1}}},{(0, -9)*{^{\eta_0}}}
\end{xy}
\end{equation}
Since the right action of $G_1$ on $E_m$ is
principal (now viewing $E_m$ as a bibundle from $G\times_M G$ to
$G$), we have an isomorphism 
\begin{equation}\label{eq:eeg} \Phi: E_m \times_{J_l, G_0\times_M G_0, J_l } E_m \cong E_m \times_{J_r,
  G_0, \bt_G} G_1. \end{equation}
The right $G_1\times_M G_1$ action is
\begin{equation}\label{eq:r-gg-act} (\gamma_3, \eta_1, \eta_0)\cdot
(\gamma_1, \gamma_2)= (\gamma_3 \cdot \gamma_1, (1,
\gamma_2^{-1})\cdot \eta_1, (\gamma_1, \gamma_2)^{-1} \cdot
\eta_0).\end{equation}
Noticing that \[J_l ( \eta_1) = J_l ((\gamma_3, 1) \eta_0)
=(\bs_G(\gamma_3), \pr_2 (J_l(\eta_0))=\pr_2(J_l(\eta_1)) ), \] we have a morphism in $\cC$,
\[ \tilde{\phi}  : (G_1\times_{G_0}(E_m\times_{J_r, G_0, e}
M))\times_{G_0\times_M G_0}E_m)  \to   G_0\times_{e\circ \bt, G_0, \bt_G}G_1,\]
by
\[ (\gamma_3, \eta_1, \eta_0)\mapsto (\bt_G(\gamma_3), \pr_G\circ
\Phi(\eta_1,(\gamma_3, 1) \eta_0) ).  \]
Further, $\tilde{\phi}$ is invariant under the right action \eqref{eq:r-gg-act} 
because the right action and left
action on a bibundle commute. Therefore, $\tilde {\phi}$ descends to a
morphism in $\cC$,
\[ \phi:((G_1\times_{G_0}(E_m\times_{J_r, G_0, e} M))\times_{G_0\times_M G_0}E_m)/G_1\times_M G_1 \to   G_0\times_{e\circ \bt, G_0, \bt_G}G_1.\]
Moreover, $\phi$ is an isomorphism by \eqref{eq:eeg} and the fact that
the first copy 
$G_1$ acts on $G_1$ by multiplication. It is not hard to check that
$\phi$ is equivariant and  commutes with the moment maps of the
bibundles. Therefore,
\[ ((G_1\times_{G_0}(E_m\times_{J_r, G_0, e} M))\times_{G_0\times_M
  G_0}E_m)/G_1\times_M G_1 \cong  G_0\times_{e\circ \bt, G_0,
  \bt_G}G_1 \]as H.S. bibundles.
One proceeds similarly to prove the other symmetric isomorphism
corresponding to $g^{-1}\cdot g \sim 1$.

Let $\varphi$ be the composed isomorphism
\begin{equation}\label{compose}
((G_1\times_{G_0}(E_m\times_{J_r, G_0, e} M))\times_{G_0\times_M
G_0}E_m)/G_1\times_M G_1 \to
((G_1\times_{G_0}E_i)\times_{G_0\times_M G_0}E_m) /G_1\times_M
G_1.
\end{equation}
Suppose $\varphi([(1_g, \eta_1, \eta_2)])=([(1_g, \teta_1,
\teta_2)])$ (we can still assume that the first component is $1$
because the $G_1\times_M G_1$ action on both sides is right
multiplication by the first copy; we can assume that they are $1$ at
the same point because $\varphi$ commutes with the moment maps on
the left leg). Examining the morphisms inside the fibre products,
we have
\[ \pr_1 \circ J_l (\eta_2)=\bt_G(1_g)=\pr_1\circ J_l(\teta_2)=g. \]
Since $\varphi$ commutes with the moment maps on the right leg, we
have
\[ J_r(\eta_2)=J_r(\teta_2).\]
Similarly to the proof of Lemma \ref{lemma:kan22}, we can show
that $G_1\times_{\bs_G, G_0, \pr_1\circ J_l} E_m$ is a Morita
bibundle from $G \times_M G$ to $G \times_M G$. Then $(1_g,
\eta_2)$ and $(1_g, \teta_2)$ are both in $G_1\times_{\bs_G, G_0,
\pr_1\circ J_l} E_m$ and their images under the right moment map
$\bs_G\times J_r$ are both $(g, J_r(\eta_2))$. By principality of
this left $G_1\times_M G_1$ action, there is a unique $(\gamma_1,
\gamma_2)\in G_1\times_M G_1$ such that
\[ (\gamma_1, \gamma_2)\cdot (1, \eta_2)=(1, \teta_2). \]
Therefore $\gamma_1=1$ and $(1, \gamma_2) \cdot \eta_2 =\teta_2$.
This left $G_1 \times_M G_1$ action on $E_m$ is exactly the left
$G_1\times_M G_1$ action on the second copy of $E_m$ in
\eqref{compose}. Using this $\gamma_2$, we have
\[ (1, \teta_1, \teta_2) \cdot (1, \gamma_2) =(1, \gamma_2^{-1})\cdot (1,\teta_1,\teta_2)= (1, \eta_1', \eta_2).\]
Therefore the isomorphism
\[\varphi: \; [(1_g, \eta_1, \eta_2)]\mapsto [(1_g, \eta'_1, \eta_2)]\]
induces a map $\psi: E_m\times_{G_0} M \to E_i$ by $ \eta_1
\mapsto \eta_1'$. It's routine to check that $\psi$ is an isomorphism
of Morita bibundles.
\end{proof}

We have seen in this lemma that the $2$-identities satisfied by
$E_i$ are actually naturally satisfied by $E_m \times_{J_r, G_0,
e}M$. Notice that for the first part of the proof, we didn't use
any information involving the inverse map.  Our conclusion is that
the inverse map represented by $E_i$ can be replaced by $E_m
\times_{J_r, G_0, e}M$ without any further conditions (not even on
the $2$-morphisms) because the natural $2$-morphisms coming along
with the bibundle $E_m \times_{J_r, G_0, e}M$ naturally go well with
the other $2$-morphisms, the $a$'s and $b$'s.

\begin{prop} \label{prop:inverse}
A stacky groupoid $\cG$ in $(\cC, \cT, \cT')$ can also be  defined by
replacing the axioms involving inverses by  the axiom 
that  
\center{$E_m \times_{J_r, G_0, e}M$ is a Morita bibundle from $G$ to
$G^{\op}$ for some good presentation $G$ of $\cG$.}
\end{prop}
\begin{proof} It is clear from Lemma \ref{isom} that the existence of the
inverse map guarantees that the bibundle  $E_m \times_{J_r, G_0,
e}M$ is a Morita bibundle from $G$ to $G^{\op}$ for a good
presentation $G$ of $\cG$.

On the other hand, if  $E_m \times_{J_r, G_0, e}M$ is a Morita
bibundle from $G$ to $G^{\op}$ for some presentation $G$ of $\cG$,
then we construct the inverse map $i: \cG \to \cG$ by this
bibundle. Because of the nice properties of  $E_m \times_{J_r,
G_0, e}M$ that we have proven in the first half of Lemma
\ref{isom}, this newly defined inverse map satisfies all the
axioms that the inverse map satisfies. \end{proof}

\begin{remark}
This theorem holds also for W-groupoids and the proof is similar.
\end{remark}

\begin{remark}
There is similar
treatment of the antipode in hopfish algebras \cite{twz}. In fact
SLie groups are a geometric version of hopfish algebras.
The geometric quotient \eqref{eq:bd-inv} corresponds to
$\hom_\cA (\bepsilon, \bDelta)$ in the case of hopfish algebra.
Thus the new definition of SLie group modulo
$2$-morphisms is analogous to the definition of hopfish algebra.
\end{remark}

Sometimes the inverse map of a stacky groupoid is given by a specific
groupoid isomorphism $i: G \to G$ on some presentation (for
example $\cG(A)$ and $\cH(A)$ in \cite{tz} and (quasi-)Hopf
algebras as the algebra counter-part).

\begin{lemma}
The inverse map of a stacky groupoid $\cG$ in $(\cC, \cT, \cT')$ is given by a groupoid
isomorphism $i: G\to G$ for some presentation $G$ if and only if
on this presentation, $E_m \times_{J_r, G_0, e}M$ is a trivial
right $G$ principal bundle over $G_0$.
\end{lemma}
\begin{proof} It follows from Lemma \ref{isom} and the fact that the
  inverse is given by a morphism $i: G\to G$ if and only if
the bibundle $E_i$ is trivial.
\end{proof}

\section{$2$-groupoids and stacky groupoids}
\subsection{From stacky groupoids to  $2$-groupoids}\label{sec:slie-2}

Suppose $\cG\rra M$ is a stacky groupoid object in $(\cC, \cT, \cT')$; in this section we
construct a corresponding $2$-groupoid object $X_2 \Rrightarrow X_1 \Rightarrow X_0$
in $(\cC, \cT'')$. When $\cG\rra M$ is an SLie groupoid,  what we
construct is a
Lie $2$-groupoid. When  $\cG\rra M$ is further a W-groupoid, the
corresponding Lie $2$-groupoid is \emph {$2$-\'etale}; that is, the maps
$X_2 \to \hom(\Lambda[2,j], X)$ are \'etale for $j=0,1,2$.

\begin{thm}\label{thm:slie-2} A  stacky  groupoid object
$\cG \rra M$ in $(\cC, \cT, \cT')$  with a good chart $G_0$ of $\cG$ corresponds to a
 $2$-groupoid object $X_2
\Rrightarrow X_1 \Rightarrow X_0$ in $(\cC, \cT'')$.

A W-groupoid with a good \'etale chart corresponds to a $2$-\'etale Lie
$2$-groupoid.
\end{thm}

\paragraph{The construction of $X_2 \Rrightarrow X_1 \Rightarrow
X_0$} \label{pa:x}
Given a stacky groupoid object $\cG\rra M$
in $(\cC, \cT, \cT') $ and a good groupoid presentation
$G_1 \rra G_0$  of $\cG$, let  $E_m$ be
the H.S. bimodule presenting the morphism $m$. Let $J_l: E_m \to
G_0\times_M G_0$ and $J_r: E_m \to G_0$ be the moment maps of the
bimodule $E_m$. Notice that for a stacky groupoid, $g\cdot 1 \simeq g $
up to a $2$-morphism; that is, $m|_{\cG \times_M M} \simeq \id$ up to a
$2$-morphism. Translating this into groupoid language,  $J_l^{-1}
(G_0\times_M M)$ and $G_1$ are the H.S. bimodules presenting $m|_{\cG
  \times_M M}$ and $\id$ respectively.
By the definition of stacky groupoids, the isomorphism is provided by $b_r:
J_l^{-1}(G_0\times_M M)\to G_1$. Similarly, we have the
isomorphism $b_l: J_l^{-1}(M \times_M G_0) \to G_1$.

We construct
\[ X_0 =M , X_1=G_0, X_2=E_m \]
with the structure maps
\begin{equation} \label{eq:stru-maps}
\begin{split}
 \d^1_0=\bs, \d^1_1=\bt: X_1\to X_0, \quad &\d^2_0=\pr_2 \circ J_l, \d^2_1=J_r, d^2_2=\pr_1\circ J_l  : X_2 \to X_1,\\
 s^0_0=e: X_0 \to X_1, \quad
 & s^1_0= b_l^{-1}\circ e_G, s^1_1= b_r^{-1}\circ e_G: X_1 \to X_2
\end{split}
\end{equation} where $\pr_i$ is the $i$-th projection
$G_0\times_M G_0 \to G_0$. Item \ref{itm:b-on-M} in Def.
\ref{def:sliegpd} implies that $s^1_0\circ s_0^0=s^1_1 \circ s^0_0$. The
other coherence conditions in \eqref{eq:face-degen} are implied by the fact that the
2-morphism  preserves moment maps.  
We still need the $3$-multiplication maps
\[
m_i:\;  \Lambda(X)_{3, i}   \to X_2 \quad i=0,\dots,3.
\]
Let us first construct $m_0$. Notice that in the $2$-associative
diagram, we have a $2$-morphism $a: m\circ(m \times \id) \to m\circ
(\id \times m)$. Translating this into the language of groupoids, we
have the following isomorphism of bimodules:
\begin{equation}\label{eq:a}
a: (( E_m \times_{G_0} G_1)\times_{G_0 \times_M G_0}E_m )/ (G_1
\times_M G_1) \to ( (G_1 \times_{G_0} E_m)\times_{G_0 \times_M
G_0}E_m)/(G_1 \times_M G_1).
\end{equation}
The plan of proof is to take $(\eta_1, \eta_2, \eta_3) \in \Lambda(X)_{3,0}$. Then
$(\eta_3, 1, \eta_1)$ represents a class in $(E_m\times_{G_0} G_1)
\times_{G_0 \times_M G_0}E_m/ {\sim}$ (we write $\sim$ when it is
clear which groupoid action  is meant). Moreover, its image under
$a$ can be represented by $( 1, \eta_0,\eta_2)$; that is,
\[ a( [(\eta_3, 1,\eta_1)] )=  [( 1, \eta_0,\eta_2)].
 \]
Then we arrive naturally at $\eta_0$.

Now we prove it strictly. To simplify our notation, we call the left and right hand sides of
\eqref{eq:a}  $L$
and $R$ respectively. Since the action on $G_1$'s is by
multiplication, we have $G_1$ principal bundles $\tilde{L}\to L=
\tilde{L}/G_1$ and $\tilde{R} \to R = \tilde{R}/G_1$, where 
\[ \tilde{L} = E_m \times_{J_r, G_0, \pr_1 J_l } E_m , \quad \tilde{R}=
E_m \times_{J_r, G_0, \pr_2 J_l } E_m,\]
with $G_1$ principal actions
\[ (\eta_3, \eta_1) \cdot \gamma' = ( \eta_3 \gamma', (\gamma',
1)^{-1} \eta_1), \quad (\eta_0, \eta_2) \cdot \gamma'= (\eta_0
\gamma', (1, \gamma')^{-1} \eta_2); \] 
they are presented by diagrams   
\[
\begin{xy}
*\xybox{(0,0);<3mm,0mm>:<0mm,3mm>::
  ,0
  ,{\xylattice{-5}{0}{-4}{0}}}="S",
{(6, 6)*{\bullet}}, {(8, 8)*{_3}},  
{(-6, 6)*{\bullet}}, {(-8, 8)*{_0}}, 
{(6, -6)*{\bullet}}, {(8, -8)*{_2}},
{(-6, -6)*{\bullet}}, {(-8, -8)*{_1}},
{(6,6)\ar@{->} (-6, 6)}, {(6, 6) \ar@{->}^{g_3} (6, -6)},
{(6, -6) \ar@{->}^{g_2} (-6, -6)}, {(6, -6) \ar@{->} (-6, 6)},
{(-6,-6) \ar@{->}^{g_1} (-6, 6)},
{(-4, -2)*{\eta_3}}, {(4, 2)*{\eta_4}},
\end{xy},
\quad
\begin{xy}
*\xybox{(0,0);<3mm,0mm>:<0mm,3mm>::
  ,0
  ,{\xylattice{-5}{0}{-4}{0}}}="S",
{(6, 6)*{\bullet}}, {(8, 8)*{_3}},  
{(-6, 6)*{\bullet}}, {(-8, 8)*{_0}}, 
{(6, -6)*{\bullet}}, {(8, -8)*{_2}},
{(-6, -6)*{\bullet}}, {(-8, -8)*{_1}},
{(6,6)\ar@{->} (-6, 6)}, {(6, 6) \ar@{->}^{g_3} (6, -6)},  {(6, 6) \ar@{->} (-6, -6)},
{(6, -6) \ar@{->}^{g_2} (-6, -6)},
{(-6,-6) \ar@{->}^{g_1} (-6, 6)},
{(-4, 2)*{\eta_2}}, {(4, -2)*{\eta_0}},
\end{xy}
\]
which all together fit inside
\[\begin{xy}
*\xybox{(0,0);<3mm,0mm>:<0mm,3mm>::
  ,0
  ,{\xylattice{-5}{0}{-4}{0}}}="S"
  ,{(-10,-10)*{\bullet}}, {(-10, -12)*{_1}},
     ,{(0,0)*{\bullet}}, {(0, 2)*{^{0}}}, {(10, -10)*{\bullet}},
     {(10, -12)*{_{2}}}, {(15, -4)*{\bullet}}, {(17,-5)*{^3}},
     {(-10, -10) \ar@{->}^{g_1} (0,0)},
     { (10, -10) \ar@{->}^{g_2} (-10,-10)},
      { (10, -10) \ar@{->} (0,0)},
     {(15, -4)\ar@{->}^{g_3} (10, -10)},
     {(15, -4)\ar@{->}(0, 0)},
     {(15, -4)\ar@{->}(-10,-10)},
\end{xy} 
\]
We imagine that the $j$-dimensional faces of the picture are
elements of $X_j$. We also put $g_i$'s in the picture to help. We view
$a: (g_1 g_2) g_3 \to g_1 (g_2 g_3)$, and $\eta_3 \in E_m$ is
responsible for $g_1 g_2$, etc.

Similarly to Lemma \ref{lemma:kan22}, $(\pr_1\circ J_l) \times J_r: E_m \to
G_0\times_{\bt, M, \bt} G_0$ is a $G$ principal bundle with left $G$
action induced by the second copy of the $G\times_M G$ bibundle action
on $E_m$. Hence we have
\[ E_m \times_{ G_0 \times_{M} G_0} E_m \cong G_1 \times_{\bs_G, G_0,
  \pr_2 \circ J_l } E_m,
\quad (\tilde{\eta}_2, \eta_2) \mapsto (\gamma, \eta_2), \quad
\text{with}\; \tilde{\eta}_2 =(1, \gamma) \eta_2 \] 
which gives rise to an isomorphism $\tilde{\phi}$ in $\cC$, 
\[
E_m \times_{J_r, G_0, \pr_2 J_l} E_m \times_{G_0 \times_M G_0} E_m \cong
E_m \times_{J_r, G_0, \pr_2 J_l} E_m \times_{\pr_2 J_l , G_0, \bs_G} G_1,
\] 
given by 
\[ \tilde{\phi}( \eta_0, \tilde{\eta}_2, \eta_2) = (\eta_0 \gamma,
\eta_2, \gamma). \]
Moreover $\tilde{\phi}$ is $G$ equivariant w.r.t.\ the following
$G$ actions  
\[ (\eta_0, \tilde{\eta}_2, \eta_2) \cdot \gamma' = (\eta_0 \gamma',
(1, \gamma')^{-1} \tilde{\eta}_2, \eta_2) , \quad (\eta_0, \eta_2,
\gamma)\cdot \gamma' = (\eta_0, \eta_2, \gamma'^{-1} \gamma). 
\]
Hence it gives an isomorphism in $\cC$ between the quotients,
\[ \phi: R \times_{G_0\times_M G_0, J_l } E_m \cong \tilde{R}. 
\] 
We have a commutative diagram 
\begin{equation} \label{eq:m-a}
\xymatrix{
\Lambda(X)_{3,0} \ar[rr]^{(\eta_1, \eta_2,
  \eta_3)\mapsto \eta_2} \ar[d]_{(\eta_1, \eta_2,
  \eta_3)\mapsto [(\eta_1, \eta_3)]} & & E_m \ar[d]^{J_l }\\
L\ar[r]^a & R \ar[r] & G_0 \times_M G_0
.}
\end{equation}
Hence there exists a morphism in $\cC$ from $\Lambda(X)_{3,0}$
to the fibre product $ R \times_{G_0\times_M G_0, J_l } E_m \cong
\tilde{R}$, and $m_0$ is defined as the composition of morphisms $\Lambda(X)_{3,0} \to \tilde{R}
\xrightarrow{\pr_1} E_m$.

For other $m$'s, we precede in a similar fashion.  More
precisely, for $m_1$ one can make the same definition as for $m_0$
but using $a^{-1}$. It is even easier to define $m_2$ and $m_3$. Thus we
realize that given any three $\eta$'s, we can always put them in
the same spots as we did for $m_0$. Then any three of them
determine the fourth. Hence the $m$'s are compatible with each
other.

\paragraph{Proof that what we construct is a $2$-groupoid}
By Prop-Def.\ \ref{def:finite-2gpd}, to show that the above construction
gives us a $2$-groupoid in $(\cC, \cT'')$, we just have to show that the $m_i$'s
satisfy the coherence conditions, associativity and the $1$-Kan and
$2$-Kan conditions. Condition $1$-Kan is implied by the fact that
$\bs,\bt: G_0\rra M$ are projections; $\Kan(2,1)$ is
implied by the fact that the moment map $J_l: E_m\to
G_0\times_{\bs,M, \bt } G_0$ is a projection;
$\Kan(2,2)$ is implied by Lemma \ref{lemma:kan22}; and  $\Kan(2,0)$ can
be proven
similarly.

\subparagraph{The coherence conditions}
The first identity in eq.\ \eqref{coco} corresponds to an identity of
$2$-morphisms,
\[ \big (1\cdot (g_2 \cdot g_3) \overset{a}{\sim} (1\cdot g_2)\cdot g_3
\sim g_2\cdot g_3 \big) = \big( 1\cdot (g_2 \cdot g_3) \sim
g_1\cdot g_2 \big).
\] More precisely, restrict the two bimodules in \eqref{eq:a} to $M \times_M G_0 \times_M
G_0$; then we get $E_m$ on the left hand side because $J_l^{-1} (M
\times_M G_0)\overset{b_l}{\cong}G_1$ and $\big( (G_1\times_M G_1)
\times_{G_0\times_M G_0}E_m \big) /G_1\times_M G_1 =E_m$. In fact, the
elements in  $(E_m
\times_{G_0}G_1)\times_{G_0\times_M G_0} E_m|_{M\times_M G_0
\times_M G_0} /{\sim}$ have the form $[(s_0\circ d_2 (\eta), 1,
\eta)]$, and  the isomorphism to $E_m$ is given by $[(s_0\circ d_2
(\eta), 1,\eta)]\mapsto \eta$. Similarly for the right hand side;
i.e., $[(s_0\circ d_1(\eta), 1, \eta)]\mapsto \eta$ gives the other
isomorphism. By \ref{itm:bl} in Def.\ \ref{def:sliegpd}, the
composition of the first and the inverse of the second map is $a$
(restricted to the restricted bimodules), so we have
\[ a ([(s_0\circ d_2 (\eta), 1,\eta )])= ([(1, \eta,s_0\circ
d_1(\eta))]),
\]
which implies the first identity in \eqref{coco}. The rest follows
similarly.

\subparagraph{Associativity} \label{sec:3-asso}
To prove the associativity, we use the cube condition
\ref{itm:a-higher} in Def.\ \ref{def:sliegpd}. Let $\eta_{ijk}$'s denote
the faces in $X_2$ fitting in diagram \eqref{eq:5gon-g}. Suppose
we are given the faces $\eta_{0i4} \in X_2$ and  the faces
$\eta_{0ij} \in X_2$. Then we have two ways to determine the face
$\eta_{123}$ using $m$'s as described in Prop-Def.\
\ref{def:finite-2gpd}. We will show below that these two
constructions give the same element in $X_2$.

Translate the cube  condition into the language of groupoids.
The morphisms become H.S. bibundles and the $2$-morphisms become the
morphisms between these bibundles. The cube condition tells us that the
following two compositions of morphisms are the same (where for
simplicity, we omit the base space of the fibre products
and the groupoids by which we take quotients):
\[
\begin{split}
     &  (E_m \times G_1 \times G_1)\times (E_m \times G_1) \times E_m /{\sim}  \quad \longleftrightarrow \quad ((g_1g_2)g_3)g_4 \\
\overset{\id\times a}{\lra}\null
& (E_m \times G_1 \times G_1)\times (G_1 \times E_m) \times E_m /{\sim}   \quad \longleftrightarrow \quad (g_1 g_2)(g_3g_4)\\
\overset{\id}{\lra}\null
& (G_1\times G_1 \times E_m) \times (E_m \times G_1)\times E_m/{\sim}  \quad \longleftrightarrow \quad (g_1 g_2) (g_3 g_4) \\
\overset{\id \times a} {\lra}\null & (G_1 \times G_1 \times E_m)\times
(G_1 \times E_m )\times E_m /{\sim}  \quad \longleftrightarrow \quad g_1(g_2(g_3
g_4))
\end{split}
\]
and
\[
\begin{split}
  & (E_m \times G_1 \times G_1)\times (E_m \times G_1) \times E_m /{\sim}  \quad \longleftrightarrow \quad ((g_1g_2)g_3)g_4 \\
\overset{a\times \id}{\lra}\null &(G_1\times E_m \times G_1) \times (E_m
\times G_1) \times E_m /{\sim}  \quad \longleftrightarrow \quad   (g_1(g_2g_3))g_4\\
\overset{\id\times a }{\lra}\null
& (G_1 \times E_m \times G_1) \times (G_1 \times E_m)\times E_m/{\sim}  \quad \longleftrightarrow \quad  g_1((g_2g_3)g_4)\\
\overset{ a \times \id} {\lra}\null & (G_1 \times G_1 \times E_m) \times
(G_1 \times E_m )\times E_m/{\sim}  \quad \longleftrightarrow \quad
g_1(g_2(g_3g_4))
.\end{split}
\]
Tracing the element $(\eta_{034}, (\eta_{023}, 1),
(\eta_{012}, 1,1))$ through the first and second composition, it
should end up as the same element. So we have
\begin{equation}\label{eq:5gon-g}
\begin{split}
  & [((\eta_{012}, 1, 1)), (\eta_{023}, 1),\eta_{034} ] \\
\overset{\id \times a}{\mapsto}\null
& [ ( (\eta_{012}, 1, 1), (1, \eta_{234}),\eta_{024} )] \\
\overset{\id}{\mapsto}\null
& [ ((1,1, \eta_{234}), (\eta_{012}, 1), \eta_{024}) ] \\
\overset{\id \times a}{\mapsto}\null & [((1,1, \eta_{234}), (1,
\eta_{124}), \eta_{014})]
,\end{split}
\quad 
\begin{xy}
*\xybox{(0,0);<3mm,0mm>:<0mm,3mm>::
  ,0
  ,{\xylattice{-5}{0}{-4}{0}}}="S"
  ,{(-10,-10)*{\bullet}}, {(-10, -12)*{_1}},
     ,
{(0,6)*{\bullet}},  {(0, 8)*{^{0}}},
{(0,0)*{\bullet}}, {(1, 1)*{^{4}}}, {(10, -10)*{\bullet}},
     {(10, -12)*{_{2}}}, {(15, -4)*{\bullet}}, {(17,-5)*{^3}},
     {(-10, -10) \ar@{->}^{g_1} (0,6)}, 
{(-10, -10) \ar@{<-} (0,0)},
     { (10, -10) \ar@{->}^{g_2} (-10,-10)},
 { (10, -10) \ar@{->} (0,6)},
      { (10, -10) \ar@{<-} (0,0)},
     {(15, -4)\ar@{->}^{g_3} (10, -10)},
{(15, -4)\ar@{->}(0,6)},
     {(15, -4)\ar@{<-}_{g_4} (0,0)},
     {(15, -4)\ar@{->}(-10,-10)},
     {(0,0)\ar@{->} (0,6)}
\end{xy} 
\end{equation}
where by definition of $m_0$, $\eta_{234}=m_0 (\eta_{034}, \eta_{024},
\eta_{023})$ and $\eta_{124}=m_0(\eta_{024}, \eta_{014}, \eta_{012})$, and
\[
\begin{split}
  &[((\eta_{012}, 1, 1), (\eta_{023}, 1), \eta_{034})]  \\
\overset{a\times \id}{\mapsto}\null
&[((1, \eta_{123}, 1), (\eta_{013}, 1), \eta_{034}) ] \\
\overset{\id \times a}{\mapsto}\null
&[((1, \eta_{123}, 1), (1, \eta_{134}), \eta_{014})] \\
\overset{a\times \id} {\mapsto}\null &[((1,1, \eta_{234}), (1,
\eta_{124}), \eta_{014})], 
\end{split}
\]
where by definition of $m_0$, $\eta_{123}= m_0(\eta_{023}, \eta_{013}, \eta_{012})$ and
$\eta_{134}= m_0(\eta_{034}, \eta_{014}, \eta_{013})$. Therefore,
the last map tells us that
\[ \eta_{123}= m_3 ( \eta_{234}, \eta_{134}, \eta_{124}). \]
Therefore associativity holds!

\subparagraph{Comments on the \'etale condition} It is easy to see
that if $G_1\rra G_0$ is an \'etale Lie groupoid, by principality of the right $G$
action on $E_m$, the moment map $E_m\to G_0\times_M G_0$ is an \'etale
Lie groupoid.
Moreover since $E_m \to \Lambda(X)_{2,j}=\Lambda[2,j](X)$ is a
surjective submersion by $\Kan(2,j)$, by dimension counting, it is
furthermore an \'etale map.

\subsection{From $2$-groupoids to stacky groupoids}\label{sec:2-slie}
If $X$ is a $2$-groupoid object in $(\cC, \cT'')$, then $G_1:= d_2^{-1}(s_0(X_0))\subset
X_2$, which is the set of bigons, is a groupoid  over
$G_0:=X_1$ (Lemma \ref{lemma:g1-g0}). Here we might notice that
there is another natural choice for the space of bigons, namely
$\tG_1:= d_0^{-1}(s_0(X_0))$. But $G_1 \cong \tG_1$ by the
following observation: given an element $\eta_3 \in G_1$, it fits
in the following picture,
\[
 \raise.4cm\hbox{
\begin{xy}
*\xybox{(0,0);<3mm,0mm>:<0mm,3mm>::
  ,0
  ,{\xylattice{-5}{0}{-4}{0}}}="S"
,{(0,0)*{\bullet}},    {(0, 2)*{^{0}}}  
,{(0,-4)*{\bullet}},   {(0,-7)*{^{1}}} 
,{(14, -8)*{\bullet}}, {(16, -10)*{^{2}}}
,{(15, -4)*{\bullet}}, {(17,-5)*{^3}}
,{(0,-4)\ar@{->}(0,0)}, 
,{(15,-4)\ar@{->} (0,-4)}, {(15, -4) \ar@{->} (0,0)}
,{(14,-8)\ar@{->}(0,0)} 
, {(14,-8)\ar@{<-}(15,-4)},{(14,-8)\ar@{->}(0,-4)},
\end{xy}} 
\mspace{-100.0mu}
 \begin{split} &\text{In this picture, $1\to
0$ and $2\to 3$ are} \\ &\text{degenerate, and $\eta_2, \eta_1$ are degenerate. }\end{split}
\]
Then $m_0$ gives a morphism 
\begin{equation}\label{eq:g-tg}\varphi: G_1 \to \tG_1
,\end{equation}
and $m_3$ gives
the inverse. Therefore we might consider only $G_1$. Then $G_1
\rra G_0$ presents a stack which has an additional groupoid
structure.

\begin{thm}\label{2-to-slie}
A $2$-groupoid object $X$ in $(\cC, \cT'')$
corresponds to a stacky groupoid object $\cG\rra X_0$ with a good chart in $(\cC, \cT,
\cT')$, 
where $\cG$ is presented by the groupoid object $G_1\rra G_0$.

A $2$-\'etale Lie $2$-groupoid corresponds to a W-groupoid with a good
\'etale chart. 
\end{thm}

We prove this theorem by several lemmas.

\paragraph{About the stack $\cG$}
\begin{lemma}\label{lemma:g1-g0}
$G_1\rra G_0$ is a  groupoid object in $(\cC, \cT'')$.
\end{lemma}
\begin{proof} The target and source maps are given by $d^2_0$ and
$d^2_1$. The identity $G_0\to G_1$ is given by $s^1_0: X_1 \to
X_2$. The image of $s^1_0$ is in $G_1 \mathrel{(\subset} X_2)$. Their
compatibility conditions are implied by the compatibility
conditions of the structure maps of simplicial manifolds. Since $G_1 $
is the pull-back of $X_2 \overset{d_1\times d_2}{\thra} X_1
\times_{d_1, X_0, d_0} X_1 $ by the map
 \[X_1 \to X_1 \times_{d_1, X_0, d_0} X_1, \quad \text{with} \; g
 \mapsto (s_0(d_1(g)), g) , \]
$\bs_G=d_1: G_1 \to  X_1$ is a surjective projection. Similarly
$\bt_G$ is also a surjective projection.

The multiplication is given by the
$3$-multiplication of $X$.
\[
 \raise.4cm\hbox{
\begin{xy}
*\xybox{(0,0);<3mm,0mm>:<0mm,3mm>::
  ,0
  ,{\xylattice{-8}{0}{-2}{0}}}="S"
,{(0,0)*{\bullet}},    {(0, 2)*{^{0}}}  
,{(-4,-2)*{\bullet}},  {(-5, -4)*{^1}}
,{(1,-5)*{\bullet}},   {(1,-8)*{^{2}}} 
,{(15, -4)*{\bullet}}, {(17,-5)*{^3}}
,{(-4,-2)\ar@{->}(0,0)}, 
,{(15,-4)\ar@{->} (-4,-2)}, {(15, -4) \ar@{->} (0,0)}
,{(1,-5)\ar@{->}(0,0)} 
, {(1,-5)\ar@{<-}(15,-4)},{(1,-5)\ar@{->}(-4,-2)},
\end{xy}} \mspace{-100.0mu}
 \begin{split} &\text{In this picture, $\eta_3=s_0 \circ s_0
     (d_0^1\circ d_2^2(\eta_2))  $ is the} \\ &\text{degenerate face
corresponding to the point $0\mathrel{(=}1=2)$. }\end{split}
\]
More precisely, any $(\eta_0, \eta_2)\in
G_1\times_{\bs_G, G_0, \bt_G} G_1$ fits in the above picture.  We
define $\eta_0\cdot \eta_2 = m_1(\eta_0, \eta_2, \eta_3)$. Then the associativity of the
$3$-multiplications ensures the associativity of ``$\cdot$''. The
inverse  is also given by $3$-multiplications: $\eta_2^{-1} =
m_0(\eta_1, \eta_2, \eta_3)$ with $\eta_1=s_0^1(d^2_1(\eta_2))$ the degenerate face in
$s^1_0(X_1)$. It is clear from the construction that all the structure maps are morphisms in
$\cC$. 
\end{proof}
\begin{remark}\label{rk:varphi}
A similar construction shows that $\tG_1\rra G_0$, with $\bt=d^1_2$ and
$\bs=d^1_1$, is a groupoid object in $(\cC, \cT'')$ isomorphic to $G_1\rra G_0$ via the
map $\varphi^{-1}$ (see equation \eqref{eq:g-tg}).
\end{remark}

\paragraph{Proof that $\cG \rra M$ is a stacky groupoid object in
  $(\cC, \cT, \cT')$}
\subparagraph{Source and target maps} There are
three maps $d^2_i: X_2\to X_1=G_0$ and they (as the moment maps of
the action) each correspond to a groupoid action. The
actions are similarly given by the $3$-multiplications as the
multiplication of $G_1$. The axioms of the actions are given by
the associativity. For example, for $d^2_1$, any $(\eta_0, \eta_2)
\in X_2 \times_{d^2_1, X_1, \bt_G}G_1$ fits inside the following picture:
\begin{equation} \label{pic:1-0}
\mspace{-100.00mu} \raise.4cm\hbox{
\begin{xy}
*\xybox{(0,0);<3mm,0mm>:<0mm,3mm>::
  ,0
  ,{\xylattice{-10}{0}{-4}{0}}}="S"
,{(0,0)*{\bullet}}, {(0, 2)*{^{0}}}
,{(0,-4)*{\bullet}},     {(0,-7)*{^{1}}}
,{(10, -10)*{\bullet}} , {(10, -12)*{_{2}}}
,{(15, -4)*{\bullet}},{(17,-5)*{^3}}
,{ (10, -10) \ar@{->} (0,0)}
,{(15, -4)\ar@{->} (10, -10)}
,{(15, -4)\ar@{->}(0, 0)}
,{(0,-4)\ar@{->}(0,0)}
,{(10,-10)\ar@{->}(0,-4)}
,{(15,-4)\ar@{->} (0,-4)}
\end{xy}} \mspace{-150.0mu}
 \begin{split} &\text{In this picture, $1\to 0$ is a degenerate edge} \\ &\text{and
$\eta_3=s_0\circ d_2^2(\eta_2)$ is a
degenerate face. }
\end{split}
\end{equation}
 Then
\begin{equation}\label{eq:g1-r-action}
\eta_0 \cdot \eta_2 := m_1(\eta_0, \eta_2, s_0d_2^2(\eta_0)).
\end{equation}
Moreover, notice that the four ways to compose source, target and
face maps $G_1 \underset{\bt_G}{ \overset{\bs_G}{\rra}} G_0
\underset{d^1_1}{\overset{d^1_0}{\rra }}X_0$ only give two
different maps: $d^1_0\bs_G$ and $d^1_1 \bt_G$. They are
surjective projections since the $d^1_i$'s, $\bs_G$ and $\bt_G$ are such,
and they give the source and target maps $\bbs, \bbt: \cG\rra X_0$
where $\cG$ is the presentable stack presented by $G_1\rra
G_0$. Therefore $\bbs$ and $\bbt$ are also surjective projections
(similarly to Lemma 4.2 in \cite{tz}). We use these two maps to form the
product groupoid
\begin{equation}\label{eq:gg}
G_1\times_{d^1_0\bs_G, X_0, d^1_1 \bt_G} G_1
\rra  G_0\times_{d^1_0,X_0,d^1_1} G_0
\end{equation}
which presents the stack
$\cG\times_{\bbs, X_0, \bbt} \cG$.

\subparagraph{Multiplication}
\begin{lemma}\label{lemma:mul}$(X_2, d^2_2\times d^2_0, d^2_1)$ is an
  H.S. bimodule from the product groupoid \eqref{eq:gg} to $G_1 \rra
G_0$.
\end{lemma}
\begin{proof}
By $\Kan(2,1)$, $d^2_2\times d^2_1$ is a surjective projection from
$X_2$ to $ G_0\times_{d^1_0,X_0,d^1_1} G_0$, so we only have to
show that the right action of $G_1\rra G_0$ on $X_2$ is free and
transitive. This is implied by $\Kan(3,j)$ and $\Kan(3,j)!$
respectively. \\
\emph {Transitivity}: any $(\eta_1, \eta_0)$ such that
$d^2_0(\eta_1)=d^2_0(\eta_0)$ and $d^2_2(\eta_0)=d^2_2(\eta_1)$
fits inside picture \eqref{pic:1-0}.
Then there exists $\eta_2:=m_2(\eta_0, \eta_1, \eta_3) \in G_1$, making $\eta_0\cdot \eta_2 =\eta_1$. \\
\emph {Freeness}: if $(\eta_0, \eta_2) \in X_2 \times_{d_1, X_1,
\bt_G}G_1$ satisfies $\eta_0\cdot \eta_2 \mathrel{(=}m_1(\eta_0, \eta_2,
\eta_3))  = \eta_0$, then $\eta_2=m_2(\eta_0, \eta_0, \eta_3)$,
and $\eta_3$ is degenerate. Thus by \ref{itm:m-iso} in Prop-Def.\ \ref{def:finite-2gpd}, $m_2 (\eta_0, \eta_0, \eta_3)$ =
$s^1_0(3\to 1)$ is a degenerate face. Therefore $\eta_2=1$.
\end{proof}

Therefore $X_2$ gives a morphism $m: \cG\times_{X_0}\cG \to \cG$.

\begin{lemma}With the source and target maps constructed above, $m$ is a multiplication of $\cG\rra X_0$.
\end{lemma}
\begin{proof} By construction, it is clear that $\bbt \circ m =
\bbt\circ \pr_1$ and $\bbs \circ m = \bbs \circ \pr_2$, where $\pr_i:
\cG\times_{\bbs, X_0, \bbt} \cG \to \cG$ is the $i$-th projection (see
the picture below).
\[ 
\begin{xy}
  *\xybox{(0,0);<3mm,0mm>:<0mm,3mm>::
  ,0
  ,{\xylattice{-5}{0}{-4}{0}}}="S"
  ,{(-10,-10)*{\circ}}, {(-12, -12)*{_1}}
     ,{(0,0)*{\bullet}}, {(0, 2)*{^{0= \bbt m = \bbt \pr_1}}}, {(10, -10)*{\bullet}},
     {(15, -12)*{_{2= \bbs m = \bbs \pr_2}}},
     {(-10, -10) \ar@{->} (0,0)}, { (10, -10) \ar@{->} (-10,-10)} , { (10, -10) \ar@{->} (0,0)}
\end{xy}
\]
To show the associativity, we reverse the argument in Section
\ref{sec:slie-2}. There, we used the $2$-morphism $a$ to
construct the $3$-multiplications. Now we use the $3$-multiplications
and their associativity to construct $a$. Given the two H.S.
bibundles presenting $m\circ (m\times \id)$ and $m\circ (\id\times
m)$ respectively, we want to construct a map $a$ as in
\eqref{eq:a}, where $E_m=X_2$ and $M=X_0$. Given any element in
$(X_2\times_{G_0} G_1) \times_{G_0\times_{X_0} G_0}X_2
/G_1\times_{X_0} G_1$, as in Section \ref{sec:slie-2}, the idea is that we can
write it in the form of $[(\eta_3, 1,\eta_1 )]$, with $(\eta_1,
\eta_2, \eta_3)\in \hom(\Lambda[3, 0], X)$ for some $\eta_2$. Then
we define
\[a([(\eta_3, 1, \eta_1 )]):= [(1,\eta_0:= m_0(\eta_1, \eta_2,
\eta_3), \eta_2 )].\] As before, we need to strictify the proof via
diagram chasing. Similarly to \eqref{eq:m-a}, we have
\begin{equation} \label{eq:a-m}
\xymatrix{
\hom(\Lambda[3,0], X) \ar[rrrr]^{(\eta_1, \eta_2,
  \eta_3)\mapsto (\eta_0, \eta_2)} \ar[d]_{(\eta_1, \eta_2,
  \eta_3)\mapsto [(\eta_1, \eta_3)]} & & & & \tilde{R} \ar[r]  \ar@{->>}[d] & E_m \ar@{->>}[d]^{J_l }\\
L\ar@{.>}[rrrr]^{a} & & & &R \ar[r] & G_0 \times_M G_0
}
\end{equation}
Hence we should show that the definition of $a$ does not depend on the
choice of $\eta_1$, $\eta_3$ and $\eta_2$ set-theoretically, and then $a$ is a morphism in
$\cC$. We first show the first statement.
First of all  (see the picture below),
\[
\raise.7cm\hbox{
\begin{xy}
*\xybox{(0,0);<3mm,0mm>:<0mm,3mm>::
  ,0
  ,{\xylattice{-5}{0}{-4}{0}}}="S"
  ,{(-10,-10)*{\bullet}}, {(-10, -12)*{_1}},{(-8,
  -14)*{\bullet}},  {(-8,
  -16)*{_{1'<1}}},
     ,{(0,0)*{\bullet}}, {(0, 2)*{^{0}}}, {(10, -10)*{\bullet}},
     {(10, -12)*{^{2}}}, {(15, -4)*{\bullet}}, {(17,-5)*{^3}},
     {(-10, -10) \ar@{->} (0,0)}, {(-10, -10) \ar@{.>} (-8,-14)},{ (10, -10) \ar@{->} (-10,-10)},
      { (10, -10) \ar@{.>} (-8,-14)} , { (10, -10) \ar@{->}
     (0,0)}, {(15, -4)\ar@{->}(10, -10)}, {(15, -4)
     \ar@{.>}(-8,-14)},{(15, -4)\ar@{->}(0, 0)},{(15, -4)\ar@{->}(-10,
     -10)}, {(-8,-14)\ar@{.>} (0,0)}
\end{xy}}
\mspace{-100.0mu}
\begin{split} &\text{In this picture, $1\to 1'$ is a degenerate edge} \\ &\text{and
$\eta_{01'1}, \eta_{1'12}$ are
degenerate faces. }  
\end{split}
\]
if we choose a different $\teta_2$, since $(\eta_1,
\eta_2, \eta_3) $ and $(\eta_1, \teta_2, \eta_3) $ are both in
$\hom(\Lambda[3,0], X)$, we have $d^2_2(\eta_2)=d^2_2(\teta_2)$
and $d^2_1(\eta_2)=d^2_1(\teta_2)$. So $\eta_2=\eta_{013}$ and
$\teta_2=\eta_{0 1' 3}$ form a degenerate horn.  By $\Kan(3,0)$ there exists
$\gamma=\gamma_{1'13}\in G_1$ such that $(1, \gamma)\cdot
\teta_2(=\gamma \cdot \teta_2)=\eta_2$, that is
$\eta_{013}=m_1(\gamma, \eta_{01'3}, s^1_1(0\to 1))$. Then by
associativity and the definition of the right $G_1$ action
\eqref{eq:g1-r-action}, we have $m_0(\eta_{023}, \eta_{01'3},
\eta_{01'2}) = \eta_{1'23}=\eta_{123} \cdot \gamma$. Therefore we
have $[(1, \eta_{1'23}, \teta_2)]=[(1, \eta_{123}, \eta_2)]$. So
the choice of $\eta_2$ will not affect the definition of $a$. Secondly  (see the
following picture),
\[
\raise.7cm\hbox{
\begin{xy}
*\xybox{(0,0);<3mm,0mm>:<0mm,3mm>::
  ,0
  ,{\xylattice{-5}{0}{-4}{0}}}="S"
  ,{(-10,-10)*{\bullet}}, {(-10, -12)*{_1}},
     ,{(0,0)*{\bullet}}, {(0, 2)*{^{0}}}
,{(1,-7)*{^{0'>0}}}
,{(0,-4)*{\bullet}}
, {(10, -10)*{\bullet}}
 , {(10, -12)*{^{2}}}
 ,    {(15, -4)*{\bullet}}
 ,{(17,-5)*{^3}}
 , {(-10, -10) \ar@{->} (0,0)},
     { (10, -10) \ar@{->}(-10,-10)},
      { (10, -10) \ar@{->} (0,0)},
     {(15, -4)\ar@{->} (10, -10)},
     {(15, -4)\ar@{->}(0, 0)},
     {(15, -4)\ar@{->}(-10,-10)},
{(0,-4)\ar@{.>}(0,0)},{(-10,-10)\ar@{.>}(0,-4)},{(10,-10)\ar@{.>}(0,-4)},{(15,-4)\ar@{.>} (0,-4)}
\end{xy} }\mspace{-100.0mu}
\begin{split} &\text{In this picture, $0'\to 0$ is a degenerate edge} \\ &\text{and
$\eta_{00'1}, \eta_{00'3}$ are
degenerate faces.}  \end{split}
\] if
we choose a different $(\teta_3=\eta_{0'12},1,
\teta_1=\eta_{0'23})$, such that 
$\eta_3=\teta_3 \cdot \gamma_{00'2}$ and
$\teta_1=(\gamma_{00'2},1)\cdot\eta_1=\gamma_{00'2}\cdot \eta_1$ for a
$\gamma_{00'2} \in G_1$, then by
associativity  we have $(\teta_1, \eta_2,
\teta_3) \in \hom(\Lambda[3,0], X)$ and
\[m_0(\teta_1, \eta_2, \teta_3) =\eta_{123}=m_0(\eta_1,\eta_2,\eta_3).\]
So this choice will not affect $a$ either. 

In all our cases, for a set-theoretical map to be a morphism in $\cC$, we only have to
verify it $\cT$-locally. Luckily, our surjective projections do have
$\cT$-local sections and  $\hom(\Lambda[3,0], X) \to L $, being a
composition of two surjective projections $\hom(\Lambda[3,0], X)=
\tilde{L} \times_{G_0 \times_M G_0} E_m \to \tilde{L} $ and $\tilde{L} \to L$, is
a surjective projection.

Now the higher coherence of $a$ follows from the associativity by
the same argument as in Section \ref{sec:3-asso}.
\end{proof}

\subparagraph{Identity} Now we notice that $s_0: X_0
\hookrightarrow G_0$ and $e_G
\circ s_0: X_0 \hookrightarrow G_1$, with $e_G$  the identity of $G$, form a groupoid morphism from
$X_0\rra X_0$ to $G_1 \rra G_0$. This gives a morphism $\bar{e}:
X_0\to \cG$ on the level of stacks.
\begin{lemma} $\bar{e}$ is the identity of $\cG$.
\end{lemma}
\begin{proof} Recall from Def.\ \ref{def:sliegpd} that we need to show that there is a $2$-morphism $b_l$
between the two maps $m\circ (\bar{e}\times \id)$ and $\pr_2: X_0
\times_{X_0, \bbt} \cG \to \cG$, and similarly a $2$-morphism $b_r$.
In our case, the H.S. bibundles presenting these two maps are
$X_2|_{X_0\times_{X_0, \bt} G_0}$ and $G_1$ respectively and they are
the same by construction, hence $b_l=\id$. For $b_r$, by Remark
\ref{rk:varphi}, we have $X_2|_{G_0\times_{\bs, X_0}X_0} = \tG_1$, so the isomorphism $
\varphi^{-1}: \tG_1 \to G_1$ is $b_r$. Item \ref{itm:b-on-M} is
implied by $s^1_0\circ s^0_0=s^1_1\circ s^0_0$.

By Remark \ref{rk:slgpd}, we only need to show item
\ref{itm:bl-br}. Translating it into the language of groupoids and bibundles, we
obtain
\[\xymatrix{
G_1\times_{X_0}X_0\times_{X_0}G_1 \ar[dd] \ar@<-1ex>[dd] & &
G_1\times_{X_0}G_1 \ar[dd] \ar@<-1ex>[dd] & & G_1 \ar[dd]
\ar@<-1ex>[dd]
\\
& \tG_1\times_{X_0}G_1 \ar[dl]\ar[dr] & & X_2 \ar[dl]\ar[dr] & \\
G_0\times_{X_0}X_0\times_{X_0}G_0 & & G_0\times_{X_0} G_0 & & G_0
\\
& G_1 \times_{X_0} G_1 \ar[ul]
\ar[ur] & & X_2 \ar[ul]
\ar[ur]\\
}
\]
Corresponding to item \ref{itm:bl-br}, we need to show that the following
diagrams commute:
$$ \xymatrix{ (\tG_1 \times_{X_0} G_1)
\mathop\times\limits_{G_0\times_{X_0} G_0}
X_2/{\sim} \ar[r]^a \ar[d]_{b_r} & (G_1\times_{X_0}
G_1 )\mathop\times\limits_{G_0\times_{X_0} G_0}
X_2 /{\sim} \ar[dl]^{b_l} \\
(G_1\times_{X_0} G_1
)\mathop\times\limits_{G_0\times_{X_0}G_0}X_2/{\sim} } $$
\[
 \xymatrix{ [(\eta_3=\eta_{012}, 1, \eta_1=\eta_{023})]
\ar@{|->}[r] \ar@{|->}[d]& [(1,  \eta_{123}=\eta_0, \eta_{013}=\eta_2)]
\ar@{|->}[dl]_{id?} \\
 [(\eta_{00'1}, 1, \eta_1 )] }.\]
Let us explain the diagram: An element $[(\eta_3, 1, \eta_1)]
\in(\tG_1 \times_{X_0} G_1) \times_{G_0\times_{X_0} G_0}
X_2/{\sim}$
fits inside picture \eqref{pic:br}\footnote{Now to avoid confusion, we call a face by its three vertices, for example now
$\eta_1=\eta_{023}$.}, 
and its image under $a$ is $[(1,
\eta_{123}, \eta_{013})]$.
\begin{equation}\label{pic:br}
\raise.7cm\hbox{
\begin{xy}
*\xybox{(0,0);<3mm,0mm>:<0mm,3mm>::
  ,0
  ,{\xylattice{-5}{0}{-4}{0}}}="S"
, {(5, 0)*{\bullet}},   {(7,2)*{^{0'>0}}}
     ,{(0,0)*{\bullet}}, {(0, 2)*{^{0}}}
,{(-8,  -14)*{\bullet}},  {(-8, -16)*{_{2}}}
, {(-10, -10) \ar@{<-} (-8,-14)}
, { (10, -10) \ar@{->} (-8,-14)} 
, {(-8,-14)\ar@{->} (0,0)}
,{(-10,-10)*{\bullet}}, {(-11, -12)*{_1}}
, {(10, -10)*{\bullet}},     {(10, -12)*{_{3}}}
, {(-10, -10) \ar@{->} (0,0)}, {(-10, -10) \ar@{->} (5,0)},{ (10, -10) \ar@{->} (-10,-10)} , { (10, -10) \ar@{->} (5,0)} , { (10, -10) \ar@{->}
     (0,0)}, {(5,0)\ar@{->} (0,0)}
\end{xy}}\mspace{-100.0mu}
\begin{split} &\text{In this picture, $0'\to 0, 2\to 1$ are
    degenerate edges} \\ &\text{and all the faces containing one of
    them} \\& \text{are degenerate except for 
$\eta_{00'1}, \eta_{012}$ and $\eta_{123}$.}  \end{split}
\end{equation} By the construction of $b_r$,
$\eta_{00'1}=b_r(\eta_{3})$. We only need to show that $\eta_{00'1}
\cdot \eta_2 = \eta_{0}\cdot \eta_1$. This is implied by the
following: We consider the 3-simplices
$(0, 0', 2, 3)$, $(0', 1, 2, 3)$ and $(0, 0', 1, 3)$, we have $\eta_1=\eta_{0'23}$, $\eta_0 \cdot
\eta_{0'23} =\eta_{0'13}$ and $\eta_{00'1} \cdot \eta_2=\eta_{0'13}$
accordingly.  

\end{proof}

\subparagraph{Inverse} By Prop.\ \ref{prop:inverse}, we only have
to show that the actions of $G$ and $G^{\op}$ on $X_2
\times_{d^2_1, G_0} M$, induced respectively by the first and
second components of the left action of $G_1 \times_M G_1 \rra
G_0\times_M G_0$, are principal (see \eqref{eq:inv-construct}). We
prove this for the first copy of $G_1$; the proof for the
second is similar. An element $(\eta_3, \eta_{1}, s_0 \circ s_0
     (d_1^1\circ d_1^2(\eta_1))) \in G_1 \times_{\bs_G,
G_0,d^2_2} X_2 \times_{d^2_1, G_0} M $ fits inside
the following picture:
\[
 \raise.4cm\hbox{
\begin{xy}
*\xybox{(0,0);<3mm,0mm>:<0mm,3mm>::
  ,0
  ,{\xylattice{-8}{0}{-2}{0}}}="S"
,{(0,0)*{\bullet}},    {(0, 2)*{^{0}}}  
,{(-4,-2)*{\bullet}},  {(-5, -4)*{^1}}
,{(15, -4)*{\bullet}}, {(17,-5)*{^2}}
,{(1,-5)*{\bullet}},   {(1,-8)*{^{3}}} 
,{(-4,-2)\ar@{->}(0,0)}, 
,{(15,-4)\ar@{->} (-4,-2)}, {(15, -4) \ar@{->} (0,0)}
,{(1,-5)\ar@{->}(0,0)} 
, {(1,-5)\ar@{->}(15,-4)},{(1,-5)\ar@{->}(-4,-2)},
\end{xy}} \mspace{-100.0mu}
 \begin{split} &\text{In this picture, $\eta_2=s_0 \circ s_0
     (d_1^1\circ d_1^2(\eta_1))  $ is the} \\ &\text{degenerate face
corresponding to the point $0\mathrel{(=}1=3)$. }\end{split}
\]
Then the freeness of the action is implied by
$\Kan!(3, 0)$, and the transitivity of the action is implied by
$\Kan(3, 0)$.

\subparagraph{Comments on the \'etale condition} If the $X_2 \to
\Lambda[2,j](X)$ are \'etale maps between smooth manifolds, then the moment map $J_l:
E_m\cong X_2 \to G_0\times_{\bs,M, \bt} G_0$ is \'etale. By
principality of the right action of the Lie groupoid $G_1 \rra G_0$, it is an
\'etale groupoid. This concludes the proof of the Theorem
\ref{2-to-slie}.

\subsection{One-to-one correspondence}

In this section, we use two lemmas to prove the following theorem:
\begin{thm}\label{thm:1-1} There is a $1$--$1$ correspondence between
 $2$-groupoid objects in $(\cC, \cT'')$
 modulo $1$-Morita
equivalence and those stacky groupoid objects in $(\cC, \cT, \cT') $ whose
identity maps have good charts.
\end{thm}

By Section \ref{sec:embedding}, good charts (respectively good
\'etale charts) always exist for SLie groupoids (respectively
W-groupoids), so we have
\begin{thm} 
There is a $1$--$1$ correspondence between Lie $2$-groupoids
(respectively $2$-\'etale Lie $2$-groupoids) modulo $1$-Morita
equivalence and SLie groupoids (respectively W-groupoids).
\end{thm}

W-groupoids are isomorphic if and only if they are isomorphic as
SLie groupoids, and $1$-Morita equivalent $2$-\'etale Lie $2$-groupoids
are $1$-Morita equivalent Lie $2$-groupoids. Therefore the \'etale
version of the theorem is implied by the general case and we only
have to prove the general case.

For the lemma below, we fix our notation: $X$ and $Y$ are
$2$-groupoid objects in $(\cC, \cT'')$ in the sense of Prop-Def.\ \ref{def:finite-2gpd};
$G_0=X_1$ and $H_0=Y_1$; $X_0=Y_0=M$. $G_1$ and $H_1$ are the
spaces of bigons in $X$ and $Y$, namely $d_2^{-1}(s_0(X_0))$ and
$d_2^{-1}(s_0(Y_0))$ respectively; both $G_1 \rra G_0$ and $H_1 \rra H_0$ are
groupoid objects, and they present presentable stacks $\cG$ and
$\cH$ respectively. Moreover $\cG \rra M$ and $\cH \rra M$ are
stacky groupoids.

\begin{lemma} \label{lemma:equi-2gpd-slie} If $f: Y\to X$ is \an \equivalence, then
\begin{enumerate}
\item \label{itm:i2s} the groupoid $H_1 \rra H_0$ constructed from $Y$ satisfies
$H_1 \cong G_1 \times_{G_0\times_M G_0} H_0\times H_0$ with the
pull-back groupoid structure (therefore $\cG\cong \cH$);
\item \label{itm:algd-morp} the above map $\bphi: \cH \cong \cG$ induces a stacky groupoid
isomorphism; that is,
\begin{enumerate}
\item \label{itm:ii2s}there are a $2$-morphism $a:\bphi \circ m_{\cH}
  \to m_{\cG} \circ (\bphi\times \bphi):
\cH\times_M\cH \to \cG$ and a
$2$-morphism $b: \bphi \circ \be_{\cH}\to \be_{\cG}: M \to \cG$;
\item \label{itm:iii2s} between maps $\cH\times_M\cH\times_M\cH \to \cG$, there is a commutative diagram of $2$-morphisms
\[
\xymatrix{
\bphi\circ m_\cH\circ (m_\cH\times \id) \ar[r]^{a_H} \ar[d]^{a} & \bphi
\circ m_\cH\circ
(\id\times m_\cH) \ar[d]^{a} \\  m_\cG\circ ( m_\cG \times \id) \circ
\bphi^{\times 3}
\ar[r]^{a_G}  & m_\cG\circ (\id\times m_\cG) \circ \bphi^{\times 3} ,}
\]
where by abuse of notation $a$ denotes the $2$-morphisms generated
by $a$ such as $a\odot (a\times \id)$;
\item \label{itm:iv2s} between maps $M\times_{M} \cH \to \cG$ and maps
$\cH\times_M M \to \cG$ there are commutative diagrams of
$2$-morphisms
\end{enumerate}
\end{enumerate}
\[
\xymatrix{\bphi \circ m_\cH\circ(\be_\cH \times \id) \ar[r]^-{b_l^H}
  \ar[d]^{a\odot b} & \bphi \circ \pr_2 \ar[d]^{\id} \\
 m_\cG \circ(\be_{\cG} \times \id) \circ (\id \times \bphi) \ar[r]^{b_l^G} & {\pr_2 \circ (\id \times \bphi) } } \quad \xymatrix{ m_\cH\circ (\id\times \be_\cH
 ) \ar[r]^-{b_r^H} \ar[d]^{a\odot b} &  \bphi \circ  \pr_1 \ar[d]^{\id} \\
 m_\cG \circ(\id\times \be_\cG) \circ (\bphi \times \id) \ar[r]^{b^G_r} & {\pr_2 \circ (\id \times \bphi)} .}
 \]
\end{lemma}
\begin{proof}
Since $Y_2\cong \hom(\partial \Delta^2, Y)\times_{\hom(\partial
\Delta^2, X)} X_2$, we have
\[
\begin{split}
H_1=d_2^{-1}(s_0(Y_0))&=d_2^{-1}(s_0(X_0)) \times_{d_1\times d_0,
X_1\times_M X_1 } Y_1\times_M Y_1 \\
&=G_1 \times_{\bt_G\times \bs_G, G_0\times_M G_0} H_0\times_M H_0 \\
&=G_1\times_{\bt_G\times \bs_G, G_0\times G_0} H_0\times H_0,
\end{split}
\]
where the last step follows from the facts that $(\bt_G \times
\bs_G)(G_1) \subset G_0\times_M G_0$ and that $f$ preserves simplicial
structures. The multiplication on $H_1$ (respectively $G_1$) is
given by $3$-multiplications on $Y_2$ (respectively $X_2$).
Therefore $H$ has the pull-back groupoid structure since $Y$ is
the pull-back of $X$. So item \ref{itm:i2s} is proven. We denote
by $\phi_i: H_i\to G_i$ the groupoid morphism. Here $\phi_0=f_1$, and
$\phi_1$ is a restriction of $f_2$.

To prove \ref{itm:ii2s}, we translate it into the following groupoid
diagram:
\[
\xymatrix{ & H_1\times_M H_1 \ar[d] \ar@<-1ex>[d] \ar[ld] & & E_{m_\cH}=Y_2 \ar[dll]\ar[drr] & & H_1 \ar[d] \ar@<-1ex>[d] \ar[ld]& \\
G_1\times_M G_1 \ar[d] \ar@<-1ex>[d] & H_0\times_M H_0  \ar[ld] & E_{m_\cG}= X_2 \ar[dll]\ar[drr] & & G_1 \ar[d] \ar@<-1ex>[d] & H_0  .\ar[ld] \\
G_0\times_M G_0  & & & & G_0 & }
\]
We need to show that the following compositions of  bibundles are
isomorphic:
\begin{equation}\label{eq:2-morp-Em}
\begin{split} &  ( Y_2\times_{G_0}G_1 ) /H_1 \mathrel{(=} \big( Y_2 \times_{H_0} H_0 \times_{G_0} G_1 \big) /H_1)
 \\
\overset{a}{\cong}\null  &H_0\times_M H_0\times_{G_0\times_M G_0} X_2
\mathrel{(=}\big( H_0 \times_M H_0 \times_{G_0\times_M G_0}  G_1\times_M G_1
\times_{G_0\times_M G_0} X_2\big) / G_1\times_M G_1) .
\end{split}
\end{equation}
By item \ref{itm:i2s}, any element in $(Y_2\times_{G_0}G_1)/H_1$
can be written as $[(\eta, 1)]$ with $\eta\in Y_2$, and we
construct $a$ by $[(\eta, 1)] \mapsto (d_2(\eta), d_0(\eta),
f_2(\eta))$. First of all, we need to show that $a$ is
well-defined. For this we only have to notice that any element in
$Y_2$ has the form 
$\eta =(\bareta, h_0, h_1, h_2)$ with $\bareta=f_2(\eta)\in X_2$
and $h_i=d_i(\eta)$, since $f_2$ preserves degeneracy maps. Also
$\gamma \in H_1$ can be written as $\gamma=( \bgamma,h_1, h'_1)$
with $\bgamma=f_2(\gamma) \in G_1$; then the $H_1$ action on $Y_2$
is induced in the following way,
\[ (\bareta, h_0, h_1, h_2) \cdot(
\bgamma,h_1, h'_1) = (\bareta \cdot \bgamma, h_0, h'_1, h_2) \]
where $h_i=d_i(\eta)$. Hence if $(\eta',1)=(\eta,1)\cdot(
\bgamma,h_1, h'_1)$, then $\bgamma=1$ and
$a([(\eta',1)])=(h_2,h_0,\bareta)=a([(\eta,1)]$. Given $(h_2, h_0,
\bareta) \in H_0\times_M H_0\times_{G_0\times_M G_0} X_2$, take
any $h_1$ such that $f_1(h_1)=d_1(\bareta)$; then $(h_0, h_1, h_2)
\in \hom(\partial\Delta^2, Y)$. Thus we construct $a^{-1}$ by
$(h_2, h_0, \bareta)\mapsto [((\bareta, h_0, h_1, h_2), 1)]$. By
the action of $H_1$, it is easy to see that $a^{-1}$ is also
well-defined. For the $2$-morphism $b$, the proof is much easier,
since in this case all the H.S. morphisms are strict morphisms of
groupoids. So we only have to use the commutative diagram
\[
\xymatrix{M\ar[r] \ar[d] & G_0 \\
H_0 \ar[ru]}
\]

Recall that the $3$-multiplications on $Y_2$ are induced from those
of $X_2$ in the following way:
\[ m^Y_0((\bareta_1, h_2, h_5, h_3), (\bareta_2, h_4, h_5, h_0), (\bareta_3, h_1, h_3, h_0))= (m^X_0(\bareta_1, \bareta_2, \bareta_3), h_2, h_4, h_1), \]
and similarly for other $m^Y$'s.
\[
\begin{xy}
*\xybox{(0,0);<3mm,0mm>:<0mm,3mm>::
  ,0
  ,{\xylattice{-5}{0}{-4}{0}}}="S"
  ,{(-10,-10)*{\bullet}}, {(-10, -12)*{_1}},
     ,{(0,0)*{\bullet}}, {(0, 2)*{^{0}}}, {(10, -10)*{\bullet}},
     {(10, -12)*{_2}}, {(15, -4)*{\bullet}}, {(17,-5)*{^3}},
     {(-10, -10) \ar@{->}^{h_0} (0,0)},
     { (10, -10) \ar@{->}^{h_1} (-10,-10)},
      { (10, -10) \ar@{->} (0,0)}, {(2, -5)*{^{h_3}}}, 
     {(15, -4)\ar@{->}^{h_2} (10, -10)},
     {(15, -4)\ar@{->}_{h_5} (0, 0)},
     {(15, -4)\ar@{.>} (-10,-10)}, {(-2,-7)*{^{h_4}}},
\end{xy} 
\]
Then we have a diagram of $2$-morphisms between composed bibundles
\[
\xymatrix{
{\big( (Y_2\times_M H_1 \times_{H_0^2} Y_2 / H_1^2)\times_{G_0}G_1 \big) / H_1} \ar[d]^{a} \ar[r]^{a_H} & {\big( (H_1\times_M Y_2 \times_{G_0^2} Y_2  \big) / H_1^2)\times_{G_0} G_1/ H_1} \ar[d]^{a} \\ \big( H_0^3\times_{G_0^3}(X_2 \times_{M}G_1 \times_{G_0^2} X_2)  \big) /G_1^2
\ar[r]^{a_G} &
\big( (H_0^3\times_{G_0^3}( G_1\times_M X_2 \times_{G_0^2} X_2 ) \big) /G_1^2
,}
\]which is
\[ \xymatrix{[(\eta_3, 1, \eta_2, 1)] \ar@{|->}[r]^{a_H} \ar@{|->}[d]^{a} & [(1, \eta_0, \eta_2, 1)] \ar@{|->}[d]^{a} \\
[(h_0, h_1, h_2, \bareta_3, 1, \bareta_1)]  \ar@{|.>}[r]^{a_G ?} &
[(h_0, h_1, h_2, 1, \bareta_0, \bareta_2)]}\] where
$\bareta_i=f_2(\eta_i)$. Here $\square^k$ denotes a suitable $k$-times fibre
product of $\square$ over $M$.  Notice that $f_2$ preserves the $3$-multiplications if and only if
$m_0(\bareta_1, \bareta_2, \bareta_3)=\bareta_0$. Since $a_G([\bareta_3, 1,
\bareta_1])=[(1, m_0(\bareta_1, \bareta_2, \bareta_3),
\bareta_2)]$, we conclude that $f_2$ preserves the $3$-multiplications  if and only
if the above diagram commutes. So \ref{itm:iii2s} is also proven.

Translating the right diagram in item \ref{itm:iv2s} into groupoid
language, we have
\begin{equation}\label{diag:a-bl-br}
\xymatrix@C=2cm{ \big( J_l^{-1}(H_0\times_M M) \times_{G_0} G_1
\big) / H_1 \ar[d]^{b^H_r} \ar[r]^{\text{restriction of
$a$}\;\;\;\;\;\;\;\;\;\;\;\;\;\;\;\;\;} &
H_0 \times_M M \times_{G_0 \times_M M} J_l^{-1}(G_0 \times_M M )/ G_1
\times_M M \ar[d]^{b^G_r} \\
(H_1\times_{G_0}G_1)/H_1 \ar[r]^{\phi_1} & H_0\times_M M
\times_{ G_0 \times_M M} G_1/G_1 \times_M M }
\end{equation}
where $J_l$ denotes the left moment map of $X_2$ or $Y_2$ to
$G_0\times_M G_0$ or $H_0\times_M H_0$. The maps are explicitly: $[(\eta, 1)]
\overset{a}{\mapsto} [(h_2, s_0(x), f_2(\eta))]
\overset{b^G_r}{\mapsto} [(h_2, s_0(x), b_r^G(f_2(\eta)))]$ and
$[(\eta, 1)] \overset{b_r^H}{\mapsto} [(b_r^H(\eta), 1)]
\overset{\phi_1}{\mapsto} [(h_2, s_0(x), \phi_1(b_r^H(\eta)))]$, where
$x=d_1(h_2)$. 
To show the commutativity of the diagram, we need to show that these two maps are the same;  that is, we need to show $b_r^G
(f_2(\eta))=f_2(b_r^H(\eta))$, since $\phi_1$ is a restriction of $f_2$. Since $b_r=\varphi^{-1}$ is
constructed by $m$'s as in Section \ref{sec:2-slie}, $f_2$
commutes with the $b_r$'s.  We have a similar diagram for the left diagram of
\ref{itm:iv2s}, which is trivially commutative since $b^{H,G}_l=\id$ by the
construction in Section \ref{sec:2-slie}. So we proved item \eqref{itm:iv2s}.
\end{proof}

To establish the inverse argument, we fix again the notation:
$\cG\rra M$ is a stacky groupoid object in $(\cC, \cT, \cT') $; $G_1\rra G_0$ and $H_1 \rra H_0$
are two groupoids  in $(\cC, \cT'') $ presenting $\cG$; $X$ and $Y$ are the
$2$-groupoids corresponding to $G$ and $H$ as constructed in Section
\ref{sec:slie-2}.

\begin{lemma}If there is a groupoid equivalence $\phi_i: H_i \to G_i$ in $(\cC, \cT'')$,
then there is a $2$-groupoid $1$-\equivalence\ $Y\to X$ in $(\cC, \cT'')$.
\end{lemma}
\begin{proof}
Since both $H$ and $G$ present $\cG$, which is a stacky groupoid
over $M$, we are in the situation described in item \ref{itm:algd-morp} in Lemma
\ref{lemma:equi-2gpd-slie}; that is, we have a $2$-morphism $a$
satisfying various commutative diagrams as in items
\ref{itm:ii2s}, \ref{itm:iii2s}, \ref{itm:iv2s}. We take
$f_0$ to be the isomorphism $M\cong M$, $f_1$ the map $\phi_0$, $f_2: Y_2\to X_2$ the map $\eta \mapsto [(\eta, 1)]
\overset{a}{\mapsto} (h_2, h_0, \bareta) \mapsto \bareta$ (see \eqref{eq:2-morp-Em}). Since
$f_2$ is made up of composition of morphisms, it is a morphism in $\cC$. Since $d_2\times d_0$ is the moment map and $a$ preserves the
moment map, we have $h_i=d_i(\eta)$ for $i=0,2$. It is clear that
$f_0$ and $f_1$ preserve the structure maps since they preserve
$\be, \bbs, \bbt$ of $\cG\rra M$. It is also clear that $d_i
f_2(\eta)=f_1(h_i)$ for $i=2,0$ since $(h_2, h_0,
\bareta=f_2(\eta)) \in H_0^2\times_{G_0^2} X_2$. Since $a$
preserves moment maps, $f_1(d_1(\eta))=J_r([(\eta,
1)])=J_r(h_2,h_0,\bareta)=d_1 (\bareta)$, where $J_r$ is the
moment map to $G_0$ of the corresponding bibundles. Hence $f_2$
preserves the degeneracy maps.

For the face maps $s^1_0, s^1_1: \square_1\to \square_2$, we recall
that $s^1_1(h)=(b^H_r)^{-1} e_H (h)$. Using  the commutative diagram
\eqref{diag:a-bl-br}, by the definition of $f_2$ and the fact that $\phi_1
e_H=e_G \phi_0$,  we have
\[ f_2(s^1_1(h))= \pr_X a([((b_r^H)^{-1} e_H(h), 1)]) = (b_r^G)^{-1} \phi_1 e_H
(h) = (b_r^G)^{-1} e_G \phi_0(h) = s^1_1 f_1(h), \]where $\pr_X$ is the natural map
$H_0^2\times_{G_0^2} X_2 \to X_2$.  We treat $s^1_0$ similarly using
the diagram for $b_l$. Hence $f_2$ preserves the face maps.

The fact that $f_2$ preserves the $3$-multiplications follows from
the proof of item \eqref{itm:iii2s} of Lemma
\ref{lemma:equi-2gpd-slie}.

Then the induced map $\phi: Y_2\to \hom(\partial \Delta^2, Y)
\times_{\hom(\partial \Delta^2, X)} X_2$ is $\eta \mapsto
[(\eta,1)]\overset{a}{\mapsto} (h_2, h_0, \bareta) \mapsto (h_0,
h_1, h_2, \bareta)$, where $\bareta=f_2(\eta)$ and
$h_i=d_i(\eta)$. As a composition of morphisms, $\phi$ is a morphism in
$\cC$.
Moreover $\phi$ is injective since $a$ is injective. For any
$(h_0, h_1, h_2,\bareta)\in \hom(\partial\Delta^2, Y)
\times_{\hom(\partial\Delta^2, X)} Y_2$, we have $(h_0, h_2,
\bareta)\in H_0\times_M H_0\times_{G_0\times_M G_0} X_2$. Then
there is an $\eta$ such that $[(\eta, 1)]=a^{-1}(h_0, h_2,
\bareta)$. Thus $\phi (\eta)=(h_0, d_1(\eta), h_2, \bareta) $,
which implies that $f_1 (d_1(\eta))=d_1(\bareta)=f_1(h_1)$.
Therefore $(1, d_1(\eta), h_1)\in H_1$ and $d_i(\eta \cdot (1,
d_1(\eta),h_1))=h_i$, $i=0,1,2$, since $d_1$ is the moment map to
$H_0$ of the bibundle $Y_2$. So $\phi(\eta \cdot (1, d_1(\eta),
h_1))= (\bareta, h_0, h_1, h_2)$, which shows the surjectivity.
Therefore $\phi$ is an isomorphism.
\end{proof}

The theorem is now proven, since we only have to consider the case
when ($1$-) Morita equivalence is given by strict ($2$-) groupoid
morphisms.

\def\cprime{$'$} \def\cprime{$'$} \def\cprime{$'$}


\end{document}